\documentclass[11pt,a4paper,oneside]{article}

\usepackage[english]{babel}
\usepackage{amsmath,amssymb,amsfonts,amsthm,mathrsfs}
\usepackage{yfonts}
\usepackage{graphicx,color}
\usepackage{epsfig}
\usepackage{multirow}
\usepackage{float}
\usepackage{hyperref}
\usepackage{enumitem}
\usepackage{pb-diagram}
\usepackage[all]{xy}
\usepackage[latin1]{inputenc}
\usepackage{amsmath,amsthm}
\usepackage[english]{babel}
\usepackage{latexsym}
\usepackage{amsfonts}
\usepackage{amssymb}
\usepackage{euscript}
\usepackage{color}
\usepackage{longtable}
\usepackage{tabularx}
\usepackage{pdfsync}
\usepackage{pstricks,pst-plot,pst-node,epsfig}
\usepackage{hyperref}
\usepackage{pgfplots}
\usepackage{multicol}
\catcode`\@=11
\def\newremark#1{\@ifnextchar[{\@orem{#1}}{\@nrem{#1}}}
\def\@nrem#1#2{%
\@ifnextchar[{\@xnrem{#1}{#2}}{\@ynrem{#1}{#2}}}
\def\@xnrem#1#2[#3]{\expandafter\@ifdefinable\csname #1\endcsname
{\@definecounter{#1}\@addtoreset{#1}{#3}%
\expandafter\xdef\csname the#1\endcsname{\expandafter\noexpand
    \csname the#3\endcsname \@remcountersep \@remcounter{#1}}%
\global\@namedef{#1}{\@rem{#1}{#2}}\global\@namedef{end#1}{\@endremark}}}
\def\@ynrem#1#2{\expandafter\@ifdefinable\csname #1\endcsname
{\@definecounter{#1}%
\expandafter\xdef\csname the#1\endcsname{\@remcounter{#1}}%
\global\@namedef{#1}{\@rem{#1}{#2}}\global\@namedef{end#1}{\@endremark}}}
\def\@orem#1[#2]#3{\expandafter\@ifdefinable\csname #1\endcsname
    {\global\@namedef{the#1}{\@nameuse{the#2}}%
\global\@namedef{#1}{\@rem{#2}{#3}}%
\global\@namedef{end#1}{\@endremark}}}
\def\@rem#1#2{\refstepcounter
      {#1}\@ifnextchar[{\@yrem{#1}{#2}}{\@xrem{#1}{#2}}}
\def\@xrem#1#2{\@beginremark{#2}{\csname the#1\endcsname}\ignorespaces}
\def\@yrem#1#2[#3]{\@opargbeginremark{#2}{\csname
         the#1\endcsname}{#3}\ignorespaces}
\def\@remcounter#1{\noexpand\arabic{#1}}
\def\@remcountersep{.}
\def\@beginremark#1#2{\rm \trivlist \item[\hskip \labelsep{\bf #1\ #2}]}
\def\@opargbeginremark#1#2#3{\rm \trivlist
        \item[\hskip \labelsep{\bf #1\ #2\ (#3)}]}
\def\@endremark{\endtrivlist}
\catcode`\@=12
\newcommand{\biindice}[3]%
{

\begin{array}[t]{c}
#1\\
{\scriptstyle #2}\\
{\scriptstyle #3}
\end{array}

}
\def\a{\alpha}
\def\b{\beta}
\def\t{\theta}

\def\s{\sigma}
\def\t{\theta}
\def\l{\lambda}

\def\o{\omega}
\def\mc{\mathcal}

\newcommand{\R}{\mathbb R}

\newcommand{\Z}{\mathbb Z}

\newcommand{\overbar}[1]{\mkern 1.5mu\overline{\mkern-1.5mu#1\mkern-1.5mu}\mkern 1.5mu}

\numberwithin{equation}{section}

\theoremstyle{definition}

\theoremstyle{plain}
\newtheorem{theorem}{Theorem}[section]
\newtheorem{proposition}{Proposition}[section]
\newtheorem{lemma}{Lemma}[section]
\newtheorem{corollary}{Corollary}[section]
\newtheorem{remark}{Remark}[section]
\newtheorem{example}{Example}[section]
\title{\vskip-2.5cm
{Planar Hamiltonian systems: index theory and applications to the existence of subharmonics}
\thanks{Work written under the auspices of the Gruppo Nazionale per l'Analisi Matematica, la Probabilit\`{a}
e le loro Applicazioni (GNAMPA) of the Istituto Nazionale di Alta Matematica (INdAM); the Ministry of
Science, Technology and Universities of Spain, under Research Grant PGC2018-097104-B-100, and of the IMI of Complutense University.
The second author, ORCID 0000-0003-1184-6231, has been also supported by contract CT42/18-CT43/18 of Complutense University of Madrid.}}

\author{
\sc
Alberto Boscaggin
\\
\small
Universit\`{a} degli Studi di Torino
\\
\small
Dipartimento di Matematica ``Giuseppe Peano''
\\
\small  Via Carlo Alberto 10, 10123 Torino, Italy. 
\\
\small E-mail: {\tt alberto.boscaggin@unito.it}
\medskip
\\
\sc Eduardo Mu\~{n}oz-Hern\'andez
\\
\small
Universidad Complutense de Madrid
\\
\small
Instituto de Matem\'{a}tica Interdisciplinar (IMI)
\\
\small
Departamento de An\'alisis Matem\'atico y Matem\'atica Aplicada
\\
\small  Plaza de las Ciencias 3, 28040   Madrid, Spain
\\
\small E-mail: {\tt eduardmu@ucm.es }
}

\date{}

\numberwithin{equation}{section}

\begin{document}

\maketitle

{

\begin{abstract}
We consider a planar Hamiltonian system of the type
$Jz' = \nabla_z H(t,z)$, 
where $H: \mathbb{R} \times \mathbb{R}^2 \to \mathbb{R}$ is a function periodic in the time variable, such that $\nabla_z H(t,0) \equiv 0$ and $\nabla_z H(t,z)$ is asymptotically linear for $\vert z \vert \to +\infty$.
After revisiting the index theory for linear planar Hamiltonian systems, by using the Poincar\'e-Birkhoff fixed point theorem we prove that the above nonlinear system has subharmonic solutions of any order $k$ large enough, whenever the rotation numbers (or, equivalently, the mean Conley-Zehnder indices) of the linearizations of the system at zero and at infinity are different. 
Applications are given to the case of planar Hamiltonian systems coming from second order scalar ODEs. 
\end{abstract}

}

%\markright{\today}

\section{Introduction}\label{sec1}
\noindent \noindent
In this paper, we deal with planar Hamiltonian systems of the type
\begin{equation}\label{hsintro}
Jz' = \nabla_z H(t,z), \qquad z \in \mathbb{R}^2,
\end{equation}
where $J$ is the standard symplectic matrix (as defined in \eqref{Jstandard}) and the Hamiltonian function $H: \mathbb{R} \times \mathbb{R}^2 \to \mathbb{R}$ is $T$-periodic in the first variable. More precisely, we are concerned with the case in which system \eqref{hsintro} can be linearized at zero and at infinity, meaning that 
$$
\nabla_z H(t,z) = S_0(t) z  + o (\vert z \vert), \qquad z \to 0, \, \text{ uniformly in } t \in [0,T],
$$
and
$$
\nabla_z H(t,z) = S_\infty(t) z  + o (\vert z \vert), \qquad \vert z \vert \to +\infty, \, \text{ uniformly in } t \in [0,T],
$$
with $S_0(t), S_\infty(t)$ continuous path of symmetric matrices.

As carefully discussed in the pioneering paper \cite{MRZ}, a powerful tool to be applied in this situation in order to investigate the $T$-periodic solvability of the system is represented by the celebrated Poincar\'e-Bikhoff fixed point theorem (we refer to \cite{DaRe02} for an introduction to this theorem and to \cite[Introduction]{FoSaZa12} for a more updated bibliography, as well as for a discussion about some controversial historical issues). Indeed, taking into account that the Hamiltonian structure of system \eqref{hsintro} yields the area-preservation property of the associated Poincar\'e map, 
a planar annulus can be naturally constructed for the application of the fixed point theorem. More precisely, one can consider an annulus centered at the origin, and having a sufficiently small inner radius and a sufficiently large outer radius: in this way, the usual boundary twist condition of the Poincar\'e-Birkhoff theorem is implied by the existence of a sufficiently large gap for the winding numbers around the origin of the solutions of system \eqref{hsintro} departing, respectively, from the inner and on the outer radius (see \cite[Th. A]{MRZ}). 
The fact that these radii are chosen small and large enough ensures that the solutions departing from them differ not too much from the solutions of the linear systems $Jz' = S_0(t)z$ and $Jz' = S_\infty(t)z$, respectively: as a consequence, the existence of $T$-periodic solutions to the nonlinear system \eqref{hsintro} can be ensured by imposing suitable conditions on (the winding number of the solutions of) the systems linearizing \eqref{hsintro} at zero and at infinity. 

As a consequence of the key \cite[Lemma 4]{MRZ}, the winding number of the solutions of a (non-resonant) linear planar Hamiltonian system can be completely characterized in terms of its so-called Conley-Zehnder index\footnote{In \cite{Abb-book,MRZ} the name Maslov index is used. Actually, the Maslov index is a more general object, defined as an intersection index for a pair of paths in the Lagrangian Grassmanian: it turns out that the Conley-Zehnder index (of a linear Hamiltonian system with periodic coefficients) can be obtained as a particular case of this general setting,  cf. \cite{CLM94}.  Nowadays, the name Conley-Zehnder index is typically preferred.}, a topological index of symplectic nature which has a crucial role when investigating periodic solutions of Hamiltonian systems (also in higher dimension) via variational methods, see \cite{Abb-book}. In such a way, a sharp comparison between the results given by the Poincar\'e-Birkhoff theorem with the ones obtained in the seminal papers \cite{AmZe80,CoZe84} is given (see again \cite{MRZ} as well as the more recent paper \cite{GiMa}, where some improvements are proposed). The particular case of planar Hamiltonian systems coming from scalar second order equations of the type $u'' + f(t,u) = 0$ is treated in detail in \cite{MaReTo14}: in this situation, the natural index to be used is the usual Morse index of the Hill's equation $u'' + q(t)u = 0$, to which the Conley-Zehnder index actually reduces (in the non-resonant setting).

Along this line of research, in this paper we use index theory to investigate the existence of \textit{subharmonic solutions} to \eqref{hsintro}, namely, solutions with minimal period $kT$, where $k \geq 2$ is an integer number. Indeed, in spite of the extensive bibliography dealing with subharmonic solutions of Hamiltonian systems (see, for instance, the bibliography in \cite{BoZa12,FoTo19}), it seems that this point of view has not been systematically developed so far and the corresponding picture appears quite fragmentary. In particular, in \cite{Abb} some results are proved, via variational methods, relying on the Conley-Zehnder indices of the linerizations at zero and at infinity; on the other hand, in the very recent paper \cite{WaQi21} the notion of rotation number (in the sense of Moser) is used and the Poincar\'e-Birkhoff is applied, under some extra assumptions. Finally, in \cite{BoFe18,BoOrZa14} some particular situations are analyzed, in the context of scalar second order equations. 

The paper is organized as follows. In Section \ref{sec2} we provide a recap of the index theory for linear planar Hamiltonian systems with periodic coefficients, with special emphasis on the relation between the Conley-Zehnder index and the winding number of the solutions. 
Even if we essentially rely on the results already given in \cite{MRZ}, our point of view is slightly different. Indeed, we move from an a priori classification for the angular behavior of the solutions of a (possibly resonant) linear Hamiltonian system, see Proposition \ref{prop2.1}: in such a way, we are naturally led to a definition of the Conley-Zehnder index for resonant systems, which differs from the classical one but seems to be more appropriate from the point of view of the rotational behavior of the solutions (and, in turn, for the application of the Poincar\'e-Birkhoff theorem to nonlinear systems). As a by-product of the analysis, the relation with the stability of the system is also presented.
Our main result of the section, Theorem \ref{th2.2}, provides a synoptical table collecting all this information: even if this result is not new in a strict sense (apart from the definition of the index for resonant sytems), we hope that our presentation and proof, which is as much self-contained as possible, can be of some interest for people working in the field. 

In Section \ref{sec3} we deal again with linear planar Hamiltonian systems, focusing our attention on two asymptotic indices
which turn out to be the crucial ones when dealing with the problem of subharmonics. 
Such indices are known as \textit{mean Conley-Zehnder index} and \textit{rotation number} and they are defined, respectively, as 
$$
m = \lim_{k \to +\infty}\frac{i_{kT}}{k} 
$$
and
$$
\rho = \lim_{k \to +\infty} \frac{\theta(kT;\omega)}{k},
$$
where $i_{kT}$ is the Conley-Zehnder index of the linear system on the interval $[0,kT]$ and $\theta(t;\omega)$ is the clockwise polar angle of a solution starting at $e^{-i\omega}$ (incidentally, it can be proved that the limit defining $\rho$ does not depend on $\omega$). The main result of the section, Theorem \ref{mequalrho}, states that these indices actually always coincide (up to a constant, precisely $m = 2\rho$): even if
this fact is somewhat expected, we do not know a reference where a proof is given (in the general setting of a linear Hamiltonian system). 
In passing, in this section we also provide some sharp iteration formula for the Conley-Zehnder index $i_{kT}$ on varying of $k$, and we again discuss the relation between the asymptotic indices $m$ and $\rho$ and the stability of the system. 

Finally, in Section \ref{sec4} we investigate the existence of subharmonic solutions to the nonlinear system \eqref{hsintro}. 
Our main result, Theorem \ref{thmain}, which unifies and extends the ones obtained in \cite{Abb,WaQi21}, 
ensures that whenever the rotation numbers (or, equivalently, the mean Conley-Zehnder indices) of the linearizations of system \eqref{hsintro} at zero and infinity are different, the existence of subharmonic solutions of order $k$ is guaranteed, for every integer $k$ sufficiently large; moreover, the number of subharmonics of order $k$ goes to infinity with $k$. Several consequences and variants of this result are then analyzed; in particular, a new application to a Lotka-Volterra system with periodic coefficients is presented. Finally, some corollaries are given for planar Hamiltonian systems coming from second order scalar equations of the type
$(\varphi(u'))' + f(t,u) = 0$, where $u \mapsto - (\varphi(u'))'$ is either the usual (one-dimensional) laplacian or a singular $\varphi$-laplacian type operator, generalizing some results previously given in \cite{BoFe18,BoGa13,BoOrZa14}.

\section{Linear Hamiltonian systems: a recap}\label{sec2}

In this section we collect some useful material about linear planar Hamiltonian systems with periodic coefficients. More precisely, after having recalled some preliminary facts about the monodromy matrix (Section \ref{sec2.1}), the notion of stability (Section \ref{sec2.2}) and the Conley-Zehnder index (Section \ref{sec2.3}), we focus on the rotational behavior of the solutions, 
by giving an exhaustive description of the possible dynamics (Section \ref{sec2.4}). Then, in Section \ref{sec2.5} we 
provide a complete analysis of the relations between winding number, Conley-Zehnder index and stability.

In principle, the content of this section should be well known; however, it seems not easy to find a reference where all these notions are carefully discussed. Thus, we have tried to make our presentation as self-contained as possible, with the hope of providing an easy-to-read and useful recap.

\subsection{Some preliminaries}
Throughout the section, we are going to consider the linear planar Hamiltonian system
\begin{equation}
\label{2.1}	
Jz'=S(t)z,\qquad z=(z_1,z_2)\in\mathbb{R}^2
\end{equation}	
where
\begin{equation}
\label{Jstandard}
J = 
\begin{pmatrix}
0 &\; -1 \\
1 &\; 0 \\
\end{pmatrix}
\end{equation}
is the standard symplectic matrix and $S(t)$ is a continuous and $T$-periodic path of symmetric $2\times2$ matrix (with $T > 0$ given). 
A solution $z(t)$ of system \eqref{2.1} is called $T$-periodic if $z(t + T) = z(t)$ for every $t \in \mathbb{R}$
and $T$-antiperiodic if $z(t + T) = -z(t)$ for every $t \in \mathbb{R}$. Notice that a $T$-antiperiodic solution is $2T$-periodic (and, of course, is not $T$-periodic unless it is the trivial solution). Finally, a solution $z(t)$ of system \eqref{2.1} is called non-trivial if $z(0) \neq 0$: by the uniqueness of the solutions for the Cauchy problems, this implies $z(t) \neq 0$ for every $t \in \mathbb{R}$.

\subsubsection{The monodromy matrix}
\label{sec2.1}
As well known, the solution $z(t;z_0)$ of the initial value problem 
\begin{equation*}
Jz'=S(t)z,\qquad z(0)=z_0\in\R^2,
\end{equation*}
can be written as
\[
z(t;z_0)=M(t)z_0,
\]
where $M(t)$ is the so-called fundamental matrix of \eqref{2.1}, that is, the matrix solving the initial value problem 
\begin{equation}
\label{2.2}
\left\{
\begin{array}{ll}
&JM'=S(t)M,\\[2pt]
&M(0)=I,
\end{array}	
\right.
\end{equation}	
with $I$ the identity matrix in $\mathbb{R}^2$. Hence, the Poincar\'{e} map at time $T$ of \eqref{2.1} is
\begin{equation*}
\begin{array}{ll}
\Psi_T:&\mathbb{R}^2\rightarrow\mathbb{R}^2\\
&z_0\;\mapsto z(T;z_0)=M(T)z_0. \end{array}
\end{equation*}
The next lemma states some basic important properties of the monodromy matrix $M(T)$ and its eigenvalues (also called \emph{Floquet multipliers}). 
\begin{lemma}
\label{le2.1}
The following hold true.
\begin{enumerate}
\item[(M1)] The characteristic polynomial of $M(T)$ is  $P_{M(T)}(\mu)=\mu^2-[{\rm tr}\,M(T)]\mu+1$. Hence, the eigenvalues of $M(T)$ belong to one of the following three classes:
\begin{equation*}
\begin{array}{ll}
(1)\;\;&\mu_1=1/\mu_2\in\mathbb{R}\setminus\{\pm1\},\\[2pt]
(2)\;\;&\mu_1=\mu_2=\pm1,\\[2pt]
(3)\;\;&\mu_1=\overbar{\mu_2}\in\mathbb{S}^1\setminus\{\pm1\},
\end{array} 	
\end{equation*}
where $\mathbb{S}^1 = \{z\in\mathbb{C}\,:\,|z|=1\}$. 	 
\item[(M2)] System \eqref{2.1} has a non-trivial $T$-periodic solution (resp., $T$-antiperiodic solution) if, and only if, $\mu_1 = \mu_2 = 1$ (resp., $\mu_1 = \mu_2 = -1$).  
\end{enumerate}
\end{lemma}	
\begin{proof}

In order to prove (M1), we first note that, since 
$S$ is symmetric, by  \eqref{2.2} the equation
\[
(M^TJM)'=(M')^TJM+M^TJM'=M^T(-S^T+S)M=0
\]
holds. Thus, as $M(0)=I$, 
\begin{equation}
\label{2.3}
M(t)^TJM(t)=M(0)^TJM(0)=J
\end{equation}
for all $t\in[0,T]$. Since
\begin{equation*}
M(t)^TJM(t) = 
\begin{pmatrix}
0 &\; -\det M(t) \\
\det M(t) &\; 0 \\
\end{pmatrix},
\end{equation*}
equality \eqref{2.3} holds if, and only if, $\det M(t)=1$ for all $t\in[0,T]$. In particular $\det M(T)=1$ and, thus, the characteristic polynomial of $M(T)$ is $P_{M(T)}(\mu)=\mu^2-[{\rm tr}\,M(T)]\mu+1$. Then, the eigenvalues $\mu_1$, $\mu_2$ are complex conjugate satisfying $\mu_1\mu_2=1$ and, hence, they belong to (1), (2) or (3). 

As for (M2) we first observe that, by the $T$-periodicity of the matrix $S(t)$, the solution $z(t;z_0)$ is $T$-periodic (resp., $T$-antiperiodic) if, and only if, $z(T;z_0) = z_0$ (resp., $z(T;z_0) = -z_0$). This means $M(T)z_0 = z_0$ (resp., $M(T)z_0 = -z_0$), and thus, for $z_0 \neq 0$, $z(t;z_0)$ is $T$-periodic (resp., $T$-antiperiodic) if and only if $z_0$ is an eigenvector of $M(T)$ with eigenvalue equal to $1$ (resp., $-1$). By (M1), in this case it must be $\mu_1 = \mu_2 = 1$ (resp., $\mu_1 = \mu_2 = -1$).
\end{proof}

\begin{remark}
\label{remsymp}
\rm 
Let us observe that the equality \eqref{2.3} actually proves that the matrix $M(t)$ is symplectic for every $t \in [0,T]$, cf. Section \ref{Sec2-1-4}. Thus, it is easily seen that the property (M1) of Lemma \ref{le2.1} holds more in general for a $2 \times 2$ symplectic matrix.
\end{remark}

According to (M2) of Lemma \ref{le2.1}, following a standard terminology from now on we will say that system \eqref{2.1} is   
$$
\left\{
\begin{array}{ll}
\hbox{\emph{hyperbolic}}\;\;&\hbox{if}\;\;\mu_1=1/\mu_2\in\mathbb{R}\setminus\{\pm1\},\\[2pt]
\hbox{\emph{parabolic}}\;\;&\hbox{if}\;\;\mu_1=\mu_2=\pm1,\\[2pt]
\hbox{\emph{elliptic}}\;\;&\hbox{if}\;\;\mu_1=\overbar{\mu_2}\in\mathbb{S}^1\setminus\{\pm1\}.
\end{array} 
\right. 	
$$
Moreover, we will say that system \eqref{2.1} is $T$-\emph{nonresonant} (or $T$-\emph{nondegenerate}) if its unique $T$-periodic solution is the trivial one, and $T$-\emph{resonant} (or $T$-\emph{degenerate}) otherwise. In view of Lemma \ref{le2.1}, we have that for a $T$-resonant system the monodromy matrix has eigenvalues 
$\mu_1 = \mu_2 = 1$. Hence, a $T$-resonant system is always parabolic (while the converse is not true).
\par
For further convenience, we now recall that system \eqref{2.1} can be transformed into an autonomous system in $\mathbb{C}^2$ via a $T$-periodic linear change of variables with complex coefficients. This is the content of Floquet's theorem (see, for example, \cite[Ch. 1, Sec. 1]{Ek} or \cite[Th. 3.4.2]{MHO}): precisely, given $M(t)$ the fundamental matrix of \eqref{2.1}, there exist a (constant) complex matrix $C:=T^{-1}\log M(T)$ (for a detailed explanation on logarithms of matrices see \cite[Sec. 4.3.2]{MHO}) and a $T$-periodic matrix $P(t)$ with complex entries such that 
\begin{equation}
\label{2.4}
M(t)=P(t)e^{tC}\qquad\hbox{for all}\;t\in\R,
\end{equation}
where $e^{tC}$ is the fundamental matrix of the constant coefficient system in $\mathbb{C}^2$
\begin{equation}
\label{2.5}
u'=Cu\,. 
\end{equation}
Since $M(0)=I$ and $P(0) = P(T)$, it follows that $M(T)=P(T)e^{TC}=e^{TC}$ and, thus,  
\begin{equation}
\label{2.6}
\s(M(T))=e^{T\s(C)}\,. 
\end{equation}
The eigenvalues of the matrix $C$ are usually called \emph{Floquet exponents}. By \eqref{2.6}, given $\mu\in\s(M(T))$ there exists $\l\in\s(C)$, such that 
\begin{equation*}
\mu=e^{T\l}\,. 
\end{equation*}
Notice that the Floquet exponents depend on the choice of $C$; however, the Floquet multipliers give the Floquet exponents modulo $2\pi i/T$.

The next section collects some classical results that allow to determine the stability of system \eqref{2.1} according to the eigenvalues of $M(T)$ and $C$.

\subsubsection{The notion of stability}\label{sec2.2}

It is said that the linear Hamiltonian system \eqref{2.1} is \emph{stable} if all its solutions are bounded on $\R$ (of course, this corresponds to the usual notion of Lyapunov stability of the origin). The next lemma is a well known result (see, e.g., \cite[Prop. 1]{Ek} or \cite[Prop. 4.2.2]{MHO}) that we include for the sake of completeness. 

\begin{lemma}
\label{le2.2}
System \eqref{2.1} is stable if, and only if, the monodromy matrix $M(T)$ is diagonalizable and its spectrum lies in $\mathbb{S}^1$. This is also equivalent to the fact that the sequence of matrices $\{M(T)^k\}_{k \in \mathbb{Z}}$ is bounded.
\end{lemma}

\begin{proof}
By \eqref{2.4}, given solutions $z(t)$ and $u(t)$ of \eqref{2.1} and \eqref{2.5}, respectively, with the same initial condition, it holds that 
\[
z(t)=P(t)u(t)\quad\hbox{for all}\;t\in\R
\]
and, hence, $z(t)$ is bounded for all $t\in\R$ if, and only if, $u(t)$ is bounded for all $t\in\R$. Moreover, the eigenspaces associated to $e^{T\l}\in\s(M(T))$ and $\l\in\s(C)$ coincide and, hence, $M(T)$ is diagonalizable if, and only if, $C$ is diagonalizable. 
\par
If $C$ is diagonalizable, there exists a basis of eigenvectors $v_1,v_2$ associated with the eigenvalues $\l_1,\l_2$, respectively, and every solution $u(t)$ of \eqref{2.5} is a linear combination of the solutions 
\[
u_1(t)=e^{\l_1 t}v_1\quad\hbox{and}\quad u_2(t)=e^{\l_2 t}v_2.
\]
Then, $u_1$ and $u_2$ are bounded for all $t$ if, and only if, the Floquet exponents $\l_1$ and $\l_2$ have zero real part, i.e., the associated Floquet multipliers $\mu_1$ and $\mu_2$ lie in $\mathbb{S}_1$. Conversely, if $C$ is not diagonalizable and $\lambda$ is its (unique) eigenvalue, there exists a solution of \eqref{2.5} of the form $u(t)=(u_1 + t u_2) e^{\l t}$ with $u_1,u_2$ suitable vectors (see \cite[Sec. III.4]{Ha}), which is not bounded on $\R$. 
\par 
On the other hand, by \eqref{2.4}, every solution $z(t)$ of system \eqref{2.1} can be written as 
\[
z(t)=M(t)z_0=P(t)M(T)^{t/T}z_0
\]
and, thus, \eqref{2.1} is stable if, and only if, $M(T)^{k}$ remains bounded as $k\in\mathbb{Z}$. 
\end{proof}

\begin{remark}
\rm In a general (possibly non-Hamiltonian) linear system a distinction can be made between \emph{positive stability} (the solutions are bounded for $t>0$) and \emph{negative stability} (the solutions are bounded for $t<0$). By redoing the proof of Lemma \ref{le2.2}, a linear  system is positive stable (resp.,  negative stable) if, and only if, the monodromy matrix is diagonalizable and its eigenvalues lie in the unit disk $\mathbb{D}:=\{z\in\mathbb{C}\,:\,|z|\leq1\}$ (respec. in the set $\overbar{\R^2\setminus\mathbb{D}}$). However, in a linear Hamiltonian system with a monodromy matrix $M(T)$, if $\mu\in\s(M(T))$, then $\mu^{-1}\in\s(M(T))$ (see (M1) in Lemma \ref{le2.1} for the planar case and \cite[Cor. 6, Ch.1, Sec. 1]{Ek} for the general case). Hence, $|\mu|<1$ if, and only if, $|\mu^{-1}|>1$ and, thus, positive and negative stability are equivalent to stability. 
\end{remark}
\par 
\noindent In a linear Hamiltonian system a further type of stability, called \emph{strong (or parametric) stability}, can be considered, as well. Precisely, system \eqref{2.1} is said to be strongly stable if there exists $\varepsilon>0$ such that all symplectic matrices $N$ satisfying	
\[
\left\lVert M(T)-N\right\rVert<\varepsilon
\]
are stable, in the sense that the sequence $\{N^k\}_{k \in \mathbb{Z}}$ is bounded.
Similarly as before, the next lemma shows the characterization of strong stability in terms of the eigenvalues of $M(T)$. 
\begin{lemma}
\label{le2.3}
System \eqref{2.1} is strongly stable if, and only if, the spectrum of the monodromy matrix $M(T)$ lies in $\mathbb{S}^1\setminus\{\pm 1\}$. 
\end{lemma}
\begin{proof}
Clearly, if $M(T)$ is strongly stable, then it is stable and hence, by Lemma \ref{le2.2}, it is diagonalizable and its spectrum lies in $\mathbb{S}^1$. Now we notice that if $M(T)$ is diagonalizable and has double eigenvalue 1 (resp. $-1$) then $M(T)=I$ (resp. $M(T)=-I$): we claim that in these cases $M(T)$ is not strongly stable. Indeed, assuming $M(T)=I$ (the case $-I$ is analogous) the matrix 
\[
N=
\begin{pmatrix}
1 &\; \varepsilon/2 \\
0 &\; 1 \\
\end{pmatrix},
\]
satisfies (in the $\infty$-norm) $\Vert M(T) - N \Vert < \varepsilon$; however, $N$ is not diagonalizable and, thus, not stable. 
\par
Conversely, if the eigenvalues $\mu_1$, $\mu_2$ of $M(T)$ belong to $\mathbb{S}^1\setminus\{\pm 1\}$, they are simple and $\mu_1=\overbar{\mu_2}$. Hence, if $N$ is a symplectic matrix which is close enough to $N$, it has eigenvalues $\lambda_1$, $\lambda_2$ close to $\mu_1$ and $\mu_2$, respectively (cf. the argument in \cite[page 7]{Ek}). This implies that $\lambda_1$ and $\lambda_2$ cannot be real and thus, since $N$ is symplectic,  they are complex, with $\lambda_1 = \overbar{\lambda_2} \neq \pm 1$. Therefore $N$ is stable.
\end{proof}

\noindent Notice that Lemma \ref{le2.3} can be generalized to the case of higher dimensional Hamiltonian systems in a sense involving the so-called Krein sign of the $M(T)$-eigenvalues  (see \cite[Th. 10, Ch. 1, Sec. 2]{Ek}). 
\par 
By Lemmas \ref{le2.2} and \ref{le2.3}, and observing (as in the proof of Lemma \ref{le2.3}) that a parabolic matrix $A$ is diagonalizable if, and only if, $A = \pm I$, the following result thus holds true.

\begin{corollary}\label{corstab}
Let us consider the linear Hamiltonian system \eqref{2.1}. Then:
\begin{itemize}
\item if \eqref{2.1} is hyperbolic, then it is unstable,
\item if \eqref{2.1} is parabolic, then it is stable when $M(T) = \pm I$ and unstable otherwise,
\item if \eqref{2.1} is elliptic, then it is strongly stable.  
\end{itemize}
\end{corollary}

\subsubsection{The Conley-Zehnder index}\label{sec2.3}
\label{Sec2-1-4}
Denoting by 
\[
Sp(1)=\{A\in\mc{M}_2(\R)\,:\,A^TJA=J\}=\{A\in\mc{M}_2(\R)\,:\, \det A=1\}
\]
the symplectic group, where $\mc{M}_2(\R)$ is the set of two by two real matrices, the Conley-Zehnder index of a path of symplectic matrices is a topological invariant that associates to each path in the set
\[
\Gamma:=\{\gamma:[0,T]\rightarrow Sp(1)\;\hbox{continuous}\,:\,\gamma(0)=I\;\hbox{and}\,1\,\hbox{is not an eigenvalue of}\,\gamma(T) \}
\]  
an integer number. A brief description of how to define the Conley-Zehnder index is the following one, cf. \cite[Sec. 2 and 3]{Abb}. 
\par
At first, we recall that every invertible matrix $A$ admits the so-called polar decomposition, i.e., $A$ can be written as $A=PO$, where $P$ is symmetric and positive definite and $O$ is orthogonal. If $A\in Sp(1)$, then $P$ belongs to the group of symmetric, positive definite and symplectic matrices, which is homeomorphic to the plane, while $O$ belongs to the group of rotations $SO(2)$, which is homeomorphic to the unit circumference $\mathbb{S}^1$. Hence, since the polar decomposition is continuous in $A$, the group $Sp(1)$ is homeomorphic to $\mathbb{R}^2\times\mathbb{S}^1$, an explicit covering projection $\Phi: \mathbb{R}^+ \times \mathbb{R} \times \mathbb{R} \to Sp(1)$ being given by
\begin{equation}\label{defcovering}
\Phi(\tau,\sigma,\vartheta) = P(\tau,\sigma)O(\vartheta),
\end{equation}
where
$$
P(\tau,\sigma) = \begin{pmatrix}
\cosh \tau + \sinh \tau \cos \sigma & \sinh\tau \sin\sigma \\
\sinh \tau \sin \sigma &\cosh\tau - \sinh\tau \cos\sigma \\
\end{pmatrix}, \qquad 
O(\vartheta) = \begin{pmatrix}
\cos \vartheta & -\sin\vartheta \\
\sin \vartheta &\cos\vartheta \\
\end{pmatrix}.
$$
Actually, since the plane is homeomorphic to the open disc, we can more conveniently visualize $Sp(1)$ as the interior of a solid torus in $\mathbb{R}^3$. 
\par
Now, since by Lemma \ref{le2.1} and Remark \ref{remsymp} any $A\in Sp(1)$ is either elliptic, parabolic or hyperbolic, a so-called rotation function $\varrho: Sp(1) \to \mathbb{S}^1$ can be defined as follows:
\begin{equation}\label{defrotfunction}
\varrho(A) = \frac{\lambda}{\vert \lambda \vert},
\end{equation}
where $\lambda$ is any eigenvalue of $A$ in the case $\lambda \in \mathbb{R}$ (that is, $\varrho(A) = \pm 1$ depending on whether the eigenvalues are both positive or both negative), and it is the Krein-positive eigenvalue of $A$ in the case $\lambda \in \mathbb{S}^1 \setminus \{\pm1\}$, that is, $\langle i J \zeta, \zeta \rangle_{\mathbb{C}^2} > 0$, where $\zeta \in \mathbb{C}^2$ is the eigenvector corresponding to $\lambda$. The name rotation function is justified by the fact that $\varrho$ is homotopic to the map $A = P(\tau,\sigma)O(\vartheta) \in Sp(1) \mapsto e^{i\vartheta}$ and actually reduces to it when $A = O(\vartheta)$, see \cite{Abb} for some details.
\par 
Based on this notion of rotation, the Conley-Zehnder index of $\gamma \in \Gamma$ is defined as follows. First, we take a path $\widetilde{\gamma}:[0,T+1]\rightarrow Sp(1)$ extending $\gamma$ (i.e., $\widetilde{\gamma}|_{[0,T]}=\gamma$) such that 
\begin{align*}
&\bullet 1 \text{ is not an eigenvalue of $\widetilde{\gamma}(t)$} \, \text{for all}\; t\in[T,T+1],\\
&\bullet\widetilde{\gamma}(T+1)\in\{N_1,N_2\},
\end{align*}
where $N_1$ (resp. $N_2$) is an arbitrary symplectic matrix with real negative (resp. positive) eigenvalues.
Second, we consider the unique function $\delta: [0,T+1] \to \mathbb{R}$ such that
\begin{equation}
\label{2.17}
e^{i\pi \delta(t)}=\varrho(\widetilde\gamma(t)), \qquad \delta(0)=0.
\end{equation}
The \textit{Conley-Zehnder index} (at time $T$) of $\gamma$ is then defined as
$$
i_T = \delta(T+1).
$$
Notice that $i_T$ is an integer number, since $\varrho(A) = \pm 1$ if $A$ is a matrix with real eigenvalues.
The Conley-Zehnder index is thus a sort of algebraic count of the rotations of the path $\gamma$ in $Sp(1)$; however, the fact of considering an extending path $\widetilde\gamma$ as above has some subtleties.
\par
\begin{figure}[h!]
\centering
\includegraphics[scale=0.8]{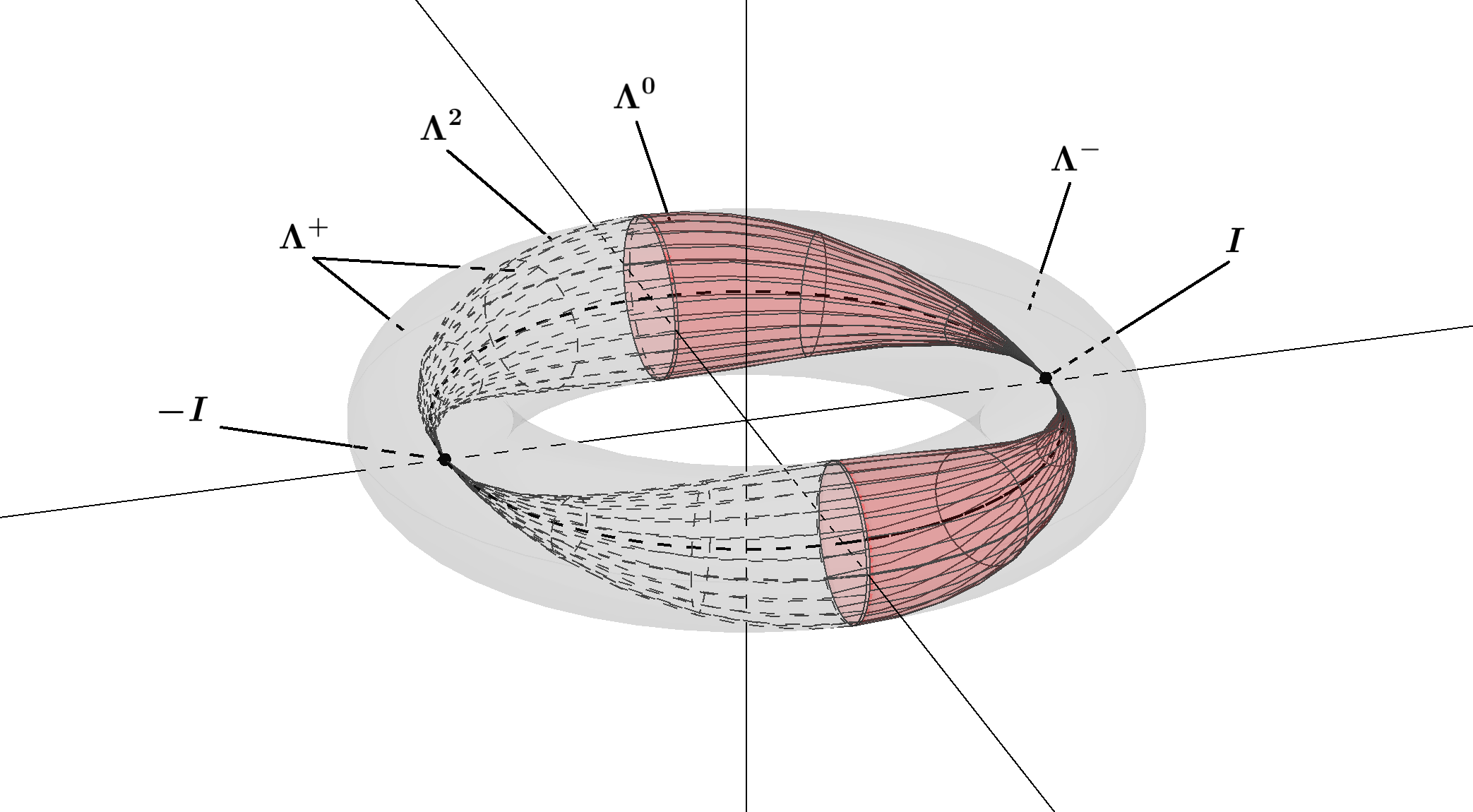}
\caption{\small This figure is a representation of the symplectic group $Sp(1)$ as the interior of a solid torus. We can distinguish different surfaces and open sets on it. At first, the red surface $\Lambda^0$ represents the so-called \emph{resonant surface}, i.e, the surface made by the symplectic matrices which have eigenvalues $\mu_1=\mu_2=1$. In particular it contains the identity matrix $I$, which is the singular point of $\Lambda^0$. This surface is the key for the definition of the Conley-Zehnder index given above, since it disconnects $Sp(1)$ in the open sets $\Lambda^-$ and $\Lambda^+$, allowing us to count the rotations of a path of symplectic matrices. Precisely, $\Lambda^-$, which is the open set in the right hand side out of $\Lambda^0$, is the set of symplectic matrices having positive real eigenvalues different from $1$; $\Lambda^2$ is the dotted surface symmetric to $\Lambda^0$ (including $-I$ as the singular point) and thus it is made by the matrices having eigenvalues $\mu_1=\mu_2=-1$ (the name $\Lambda^2$ is chosen to emphasize that this is a $2T$-resonant surface). The open set enclosed by $\Lambda^0\cup\Lambda^2$ contains the matrices in $Sp(1)$ with complex eigenvalues, while the open set in the left hand side out of $\Lambda^2$ includes the symplectic matrices having negative real eigenvalues different from $-1$. Hence, $\Lambda^+=Sp(1)\setminus\{\Lambda^0\cup\,\Lambda^-\}$ contains all the matrices in $Sp(1)$ having complex or negative real eigenvalues. 
Notice that both $\Lambda^-$ and $\Lambda^+$ are contractible in $Sp(1)$. Finally, the horizontal circumference $\mathbb{S}^1$ passing trough $I$ and $-I$, drawn with a dashed line, includes all the symplectic matrices which are rotation matrices.}
\label{Fig1}
\end{figure}
To get a better geometrical insight of the definition, it is convenient to define
\begin{align*}
\Lambda^-&= \{A\in Sp(1)\,:\,\det(I-A)<0\} = \{A\in Sp(1)\,:\,\l_1,\l_2\in\R^+\setminus\{1\}\},\\
\Lambda^0&= \{A\in Sp(1)\,:\,\det(I-A)=0\} = \{A\in Sp(1)\,:\,\l_1=\l_2=1\}\\
\Lambda^+&= \{A\in Sp(1)\,:\,\det(I-A)>0\} = \{A\in Sp(1)\,:\,\l_1,\l_2\in\R^-\cup(\mathbb{S}^1\setminus\{1\})\},
\end{align*}
where the equalities easily follow from the relation 
$$
2-\det(I-A) =  \textrm{tr}(A) = 
\begin{cases}
\lambda + 1/\lambda & \text{ if  } \lambda,1/\lambda \text{ are the real eigenvalues of } A \\
2\,\text{Re}(\lambda) & \text{ if } \lambda, \bar\lambda \text{ are the complex eigenvalues of } A,
\end{cases}
$$
valid for any $A \in Sp(1)$. A visual representation of the above sets is given in Figure \ref{Fig1}; in particular, we observe that
the surface $\Lambda_0$ disconnects $Sp(1)$ in the two connected open subsets 
$\Lambda^-$ and $\Lambda^+$, both contractible in $Sp(1)$. Hence, the extending path $\widetilde\gamma$ is such that
$\widetilde\gamma(t) \in \Lambda^-$ (resp., $\Lambda^+$) for every $t \in [T,T+1]$ if $\gamma(T) \in \Lambda^-$ (resp., $\Lambda^+$); since these sets are contractible in $Sp(1)$, the definition of Conley-Zehnder index is well posed. Incidentally, let us also notice that from this discussion it also easily follows that
$$
i_T\in 2\Z \Longleftrightarrow \gamma(T)\in \Lambda^- \qquad\hbox{ and }\qquad
i_T\in 2\Z+1 \Longleftrightarrow  \gamma(T)\in \Lambda^+.
$$  
Now, let us take $N_1$ (resp. $N_2$) a symplectic matrix with negative (resp. positive) eigenvalues and $\vartheta=\pi$ (resp. $\vartheta=0$) in the parameterization of $Sp(1)$ given by \eqref{defcovering}. With this choice, and recalling the definition of the rotation function $\varrho$, it is not difficult to convince oneself that $i_T$ is nothing but the algebraic count of the half-windings of the extended path 
$\widetilde\gamma$ in the symplectic group.
\par
As pointed out in \cite{MRZ}, using the above description of $Sp(1)$ it is also possible to give an even simpler characterization of the Conley-Zehnder index, avoiding at all the extending path $\widetilde\gamma$. Precisely, by \cite[Formulas 18 and 19]{MRZ} it holds that
\begin{equation}\label{formulaMRZ1}
i_T = 2\ell \Longleftrightarrow 
\begin{cases}
\displaystyle \gamma(T) \in \Lambda^- \quad \vspace{0.2cm}\\ 
\displaystyle 2\ell \pi - \frac{\pi}{2} < \vartheta(T) < 2\ell\pi + \frac{\pi}{2}, 
\end{cases}
\end{equation}
and
\begin{equation}\label{formulaMRZ2}
i_T = 2\ell+1 \Longleftrightarrow
\begin{cases}
\displaystyle \gamma(T) \in \Lambda^+ \vspace{0.2cm}\\  
2\ell \pi < \vartheta(T) < 2(\ell +1) \pi.
\end{cases}
\end{equation}
We will make use of these formulas in Section \ref{sec3}. 
\par 
Let us now consider the linear Hamiltonian system \eqref{2.1}.
By \eqref{2.3}, the fundamental matrix $M(t)$ belongs to $Sp(1)$ for all $t\in[0,T]$ and we can thus define the Conley-Zehnder index (at time $T$) of \eqref{2.1} as the Conley-Zehnder index $i_T$ of $M(t)$. Notice, however, that in principle we need to assume that system \eqref{2.1} is $T$-nonresonant, in order to ensure that $M(t) \in \Gamma$. In the next section, we will give a definition of the Conley-Zehnder index for a possibly resonant system: as pointed out in \cite{Abb-book}, ``the extensions to degenerate paths are somehow arbitrary'' and different extensions could be particularly appropriate for specific purposes. In fact, our definition of the Conley-Zehnder index for a resonant system will be the natural from the point of view of the rotation behavior of its solutions, see Remark \ref{remdefres} and Section \ref{secMorse}.

\subsubsection{The rotational behavior of the solutions: a classification}\label{sec2.4}

In order to analyze the winding number of the solutions of \eqref{2.1} solutions we use clockwise polar coordinates.
By the linearity of the system, it is enough to assume $z_0 = e^{-i\omega}$ and we thus write
\[
z(t;e^{-i\omega})=(z_1(t;e^{-i\omega}),z_2(t;e^{-i\omega}))=r(t;\omega)(\cos\t(t;\o),-\sin\t(t;\o))=r(t;\omega)e^{-i\t(t;\o)},
\]
that is, $r(t;\omega)>0$ and $\t(t;\o)$ are, respectively, the radial and the angular component of the \eqref{2.1} solution with initial value $z_0=e^{-i\o}$. Without loss of generality, we take $\t(0;\o) = \o$. 
\par
We thus define the (clockwise) \emph{winding number} (at time $T$) of the solution starting at $e^{-i\omega}$
as
\begin{equation}
\label{2.9}
\eta_T(\o)=\frac{\t(T;\o)-\o}{2\pi}.
\end{equation}
The next lemma states three fundamental properties of $\eta_T$.
\begin{lemma}
\label{le2.iii}
The next three properties hold.
\begin{enumerate}
\item[($W1$)] $\eta_T(\o)\in\mathbb{Z}$ (resp. $\eta_T(\o)\in\mathbb{Z}+1/2$) if, and only if, $e^{-i\o}$ is an eigenvector of $M(T)$ with a positive (resp. negative) eigenvalue. 
\item[($W2$)] The solution $z(t;e^{-i\o})$ of \eqref{2.1} is $T$-periodic if, and only if, 
\[
\eta_T(w)\in\mathbb{Z}\quad\hbox{and}\quad \eta_T'(\o)=0.
\]
and $T$-antiperiodic if, and only if, 
\[
\eta_T(w)\in\mathbb{Z}+ \frac{1}{2} \quad\hbox{and}\quad \eta_T'(\o)=0.
\]
\item[($W3$)] The function $\eta_T(\o)$ has period $\pi$. 
\end{enumerate}
\end{lemma}	
\begin{proof}
Since 
$$
M(T) e^{-i\o} = r(T;\o)e^{-i\t(T;\o)} = r(T;\o) e^{-2\pi i \eta_T(\o)} e^{-i\o},
$$
property (W1) follows plainly.
In order to prove (W2), we first observe that the solution $z(t;e^{-i\o})$ of \eqref{2.1} is $T$-periodic (resp., $T$-antiperiodic) if, and only if, $\eta_T(\o)\in\Z$ (resp., $\eta_T(\o) \in \mathbb{Z} + 1/2$) and $r(T;\omega)=1$. Then, we should show that the condition on the radial component is equivalent to $\eta_T'(\o)=0$. First, the symmetric matrix $S(t)$ of \eqref{2.1} can be written as  
\begin{equation*}
S(t) = 
\begin{pmatrix}
a(t) &\; b(t) \\
b(t) &\; c(t) \\
\end{pmatrix}
\end{equation*}
for $a(t)$, $b(t)$, $c(t)\in\mathcal{C}([0,T];\R)$. 
Rewriting \eqref{2.1} in coordinates 
\begin{equation*}
\left\{
\begin{array}{ll}
z_1'&=b(t)z_1+c(t)z_2\\
z_2'&=-a(t)z_1-b(t)z_2	
\end{array}
\right.	
\end{equation*}	
it follows that $\t(t;\o)$ solves 
\begin{equation}
\begin{split}
\label{2.11}
\t'(t;\o)&=\frac{\langle Jz',z\rangle}{|z|^2}=\frac{\langle S(t)z,z\rangle}{|z|^2}=\frac{a(t)z_1^2+c(t)z_2^2+2b(t)z_1z_2}{r^2(t;e^{i\o})}\\[2pt]
&=a(t)\cos^2\t(t;\o)+c(t)\sin^2\t(t;\o)-2b(t)\cos\t(t;\o)\sin\t(t;\o).
\end{split}
\end{equation}
In order to lighten the notation we denote $\t(t;\o)=\t$ and $r(t;e^{-i\o})= r$. Now, differentiating $\eqref{2.11}$ with respect to $\o$,
\begin{equation}
\label{2.12}
\left( \frac{\partial\t}{\partial\o}\right)' =2[\sin\t\cos\t(-a(t)+c(t))-(\cos^2\t-\sin^2\t)b(t)]\frac{\partial\t}{\partial\o}.
\end{equation}
In addition, since $r^2= z_1^2+z_2^2$, the equation
\begin{equation}
\label{2.13}
-2\frac{r'}{r}=2[\sin\t\cos\t(-a(t)+c(t))-(\cos^2\t-\sin^2\t)b(t)]
\end{equation}
holds and, thus, joining \eqref{2.12} and \eqref{2.13}, we get
\begin{equation*}
\left( \frac{\partial\t}{\partial\o}\right)' =	-2\frac{r'}{r}\frac{\partial\t}{\partial\o}\,.
\end{equation*}
Since $\frac{\partial\t}{\partial\o}(0) = 1$, we thus find
\begin{equation}
\label{2.14}
\frac{\partial\t}{\partial\o}=\frac{1}{r^2}\qquad\hbox{for all}\;\,t\in[0,T].
\end{equation}
Thus, it follows by \eqref{2.14} that
\[
\eta_T'(\o)=\frac{1}{2\pi}\left(\frac{1}{r^2(T;e^{-i\o})}-1\right).
\]
Therefore, $\eta_T'(\o)=0$ if, and only if,  $r(T;e^{-i\o})=1=r(0;e^{-i\o})$. This concludes the proof of (W2).

As for (W3) we observe that, since the right-hand side of equation \eqref{2.11} is $\pi$-periodic in the variable $\omega$, it holds
\begin{equation}\label{thetaomega}
\theta(t,\o + \pi j) = \theta(t,\o) + \pi j, \quad \text{ for every } t,\o \in \mathbb{R}, \, j \in \mathbb{Z}.
\end{equation}
Hence, 
\[
\eta_T(\o+\pi)=\frac{\t(T;\o+\pi)-(\o+\pi)}{2\pi}=\frac{\t(T;\o)+\pi-(\o+\pi)}{2\pi}=\eta_T(\o),\] 
as desired.
\end{proof}
\noindent Taking into account (W3) of the previous lemma, we now define 
\begin{equation}
\label{minmax}
\eta_T^-=\min_{\o\in[0,\pi)}\eta_T(\o)\quad\hbox{and}\quad\eta_T^+=\max_{\o\in[0,\pi)}\eta_T(\o). 
\end{equation}
Based on Lemma \ref{le2.iii}, the next lemma
provides a classification for system \eqref{2.1} in terms of the winding number $\eta_T$ of the solutions.   

\begin{proposition}
\label{prop2.1}
There exists $\ell \in \Z$ such that the winding number $\eta_T$ associated with system \eqref{2.1} satisfies one, and only one, of the following conditions:
\begin{itemize}
\item[(1)]  $\ell-1/2<\eta_T^-\leq\ell\leq\eta_T^+<\ell+1/2$; precisely
\begin{align*}
&(1.\textsc{a})\;\,\ell-1/2<\eta_T^{-}<\ell<\eta_T^+<\ell+1/2,\quad &&(1.\textsc{b})\;\,\ell-1/2<\eta_T^-<\eta_T^+=\ell,\\
&(1.\textsc{c})\;\,\ell=\eta_T^-<\eta_T^+<\ell+1/2,\quad &&(1.\textsc{d})\;\,\ell=\eta_T^-=\eta_T^+.
\end{align*}
\item[(2)] $\ell<\eta_T^-\leq\eta_T^+<\ell+1$; precisely
\begin{align*}
&(2.\textsc{a})\;\,\ell<\eta_T^-<\ell+1/2<\eta_T^+<\ell+1,\quad
&&(2.\textsc{b})\;\,\ell<\eta_T^-<\eta_T^+=\ell+1/2,\\
&(2.\textsc{c})\;\,\ell+1/2=\eta_T^-<\eta_T^+<\ell+1,\quad
&&(2.\textsc{d})\;\,\ell+1/2=\eta_T^-=\eta_T^+,\\
&(2.\textsc{e})\;\,\ell<\eta_T^-\leq\eta_T^+<\ell+1/2,\quad
&&(2.\textsc{f})\;\,\ell+1/2<\eta_T^-\leq\eta_T^+<\ell+1.
\end{align*}
\end{itemize}
Moreover:
\begin{itemize}
\item system \eqref{2.1} admits nontrivial $T$-periodic solutions if, and only if, we are in the cases (1.\textsc{b}), (1.\textsc{c}) or (1.\textsc{d}),  
\item system \eqref{2.1} admits nontrivial $T$-antiperiodic solutions if, and only if, we are in the cases (2.\textsc{b}), (2.\textsc{c}) or (2.\textsc{d}).  
\end{itemize}
\end{proposition}

\begin{proof}
Since we are in $\mathbb{R}^2$, the dimension of the eigenspace associated to an eigenvalue of $M(T)$ is always one or two. Then, assuming that $\eta_T$ is non constantly equal to $\ell$ or $\ell+1/2$ (i.e., we are not in the cases (1.\textsc{d}) and (2.\textsc{d})), by (W1) in Lemma \ref{le2.iii}, it follows that
\begin{equation*}
\kappa:={\rm {card}}\{\eta_T([0,\pi))\cap[\Z\cup(\Z+1/2)]\}\leq 2. 
\end{equation*}
If $\kappa=0$, we are in the cases (2.\textsc{e}) or (2.\textsc{f}). If $\kappa=1$, by (W3) in Lemma \ref{le2.iii}, there exist $\t_1\in[0,\pi)$ and $\ell\in\Z$ such that 
\[
\eta_T(\t_1)=\ell\in\{\eta_T^-,\eta_T^+\}\quad\hbox{or}\quad \eta_T(\t_1)=\ell+1/2\in\{\eta_T^-,\eta_T^+\}.
\]
Indeed, if $\eta_T(\t_1)$ is not a minimum or a maximum, by the $\pi$-periodicity of $\eta_T$ there exists another $\theta_2\in[0,\pi)$ such that $\eta(\t_2)\in\{\ell,\ell+1/2\}$, contradicting the fact that $\kappa=1$. Thus, if $\kappa=1$, we are in the cases (1.\textsc{b}), (1.\textsc{c}), (2.\textsc{b}) or (2.\textsc{c}). Finally, if $\kappa=2$, again by the periodicity of $\eta_T$, there exist $\t_1,\t_2\in[0,\pi)$, with $\t_1 \neq t_2$, such that 
one of the following three alternatives holds:
\begin{equation}
\label{2.15}
\eta_T(\t_1)=\eta_T(\t_2)\in\Z\qquad \mbox{ or } \qquad  
\eta_T(\t_1)=\eta_T(\t_2)\in\Z+ \frac{1}{2},
\end{equation}
or
\begin{equation}
\label{2.16}
\eta_T(\t_1)\in\Z\quad\hbox{and}\quad\eta_T(\t_2)\in\Z+\frac{1}{2}.
\end{equation}
In the two cases of \eqref{2.15} we are in (1.\textsc{a}) or (2.\textsc{a}). If  \eqref{2.16} holds, by (W1) in Lemma \ref{le2.iii} the monodromy matrix $M(T)$ has a positive and a negative eigenvalue, which cannot happen by (M1) in Lemma \ref{le2.1}. This concludes the proof of the first claim of the lemma. 
\par 
By (W2) in Lemma \ref{le2.2}, cases (2.\textsc{e}) and (2.\textsc{f}) do not admit nontrivial $T$-periodic or $T$-antiperiodic solutions since $\eta_T(\omega)\notin\Z/2$ for all $\o\in[0,\pi)$. In the cases (1.\textsc{d}) (resp. (2.\textsc{d})) all the solutions are $T$-periodic (resp. $T$-antiperiodc). On the other hand, 
as already observed in the first part of the proof in the cases (1.\textsc{b}) and (1.\textsc{c}) (resp. (2.\textsc{b}) and (2.\textsc{c})) there exists $\t_1\in[0,\pi)$ (resp. $\t_2\in[0,\pi)$) such that 
\[
\eta_T(\t_1)\in\Z\quad\hbox{and}\quad\eta_T'(\t_1)=0\qquad (\hbox{resp.}\;\, \eta_T(\t_2)\in\Z+1/2\quad\hbox{and}\quad\eta_T'(\t_2)=0 ),
\]
which implies, by (W2), the existence of a $T$-periodic solution (resp. a $T$-antiperiodc solution). Finally, we analyze the cases (1.\textsc{a}) and (2.\textsc{a}). We will prove that they are $T$-nonresonant and $2T$-nonresonant, respectively. In the case (1.\textsc{a}) there exist $\t_1,\t_2\in[0,\pi)$, with $\t_1\neq\t_2$, and $\ell\in\Z$ such that $\eta_T(\t_1)=\ell=\eta_T(\t_2)$. Then, by (W1) in Lemma \ref{le2.iii} the vectors $e^{-i\t_1}$ and $e^{-i\t_2}$ are both eigenvectors of the monodromy matrix $M(T)$ with an associated positive eigenvalue. If system \eqref{2.1} is $T$-resonant, then the only eigenvalue is $1$, and so $z(t;e^{-i\t_1})$ and $z(t;e^{-i\t_2})$ are eigenvectors associated with the eigenvalue $1$, that is, they are $T$-periodic solutions. These $T$-periodic solutions are linearly independent and have the same winding number $\eta_T=\ell$. Therefore, every solution is $T$-periodic, so that by (W2) in Lemma \ref{le2.iii} 
$\eta_T$ must be constantly equal to $\ell$, which is a contradiction. Then the case (1.\textsc{a}) is $T$-nonresonant. With an analogous proof the case (2.\textsc{a}) is $2T$-nonresonant. This ends the proof. 
\end{proof}
\noindent Note that, thanks to Proposition \ref{prop2.1}, we can sharpen (W3) of Lemma \ref{le2.iii} by saying that $\eta_T$ has minimal period $\pi$ in the cases (1.\textsc{a}), (1.\textsc{b}), (1.\textsc{c}) and (2.\textsc{a}), (2.\textsc{b}), (2.\textsc{c}). Indeed, if $\eta_T$ has a smaller period in these cases, then the number of intersections at levels $\ell\in\Z/2$ do not correspond with the number $\kappa$ shown in the proof of Proposition \ref{prop2.1}.

\subsection{Rotation, Conley-Zehnder index and stability: the relations}\label{sec2.5}

Based on the material collected in the previous sections, we now describe, for the linear Hamiltonian system \eqref{2.1}, the relations between the winding number $\eta_T$ of the solutions, the Conley-Zehnder index $i_T$ and the stability of the system. 
\par
In the next theorem we denote $\R^+_{*}=\{x\in\R\setminus\{1\}\,:\,x>0\}$, $\R^-_{*}=\{x\in\R\setminus\{-1\}\,:\,x<0\}$
and $\mathbb{S}^1_* = \mathbb{S}^1 \setminus \{\pm 1\}$. Moreover, the letters \emph{h, p, e} appearing on the first columns of the tables stand for hyperbolic, parabolic and elliptic, respectively. Finally, the subscript \emph{r}, when present, denotes that the system is $T$-resonant, while the superscripts $\{+,-,*\}$ act as labels to
distinguish among the possible cases (being motivated by the behavior of the winding number). 

\begin{theorem}
\label{th2.2}
Let $\ell \in \Z$. With reference to the classification provided by Proposition \ref{prop2.1}, the following hold true: 

\begin{center}
\begin{tabular}{| l | c | c | l | l |}
\hline
\multicolumn{5}{|c|}{Case 1: $\ell-1/2<\eta_T^-\leq\ell\leq\eta_T^+<\ell+1/2$} \\ \hline &
$\eta_T$ & $i_T$  & Eigenvalues & Stability \\ \hline
$h$ & $\ell-1/2<\eta_T^{-}<\ell<\eta_T^+<\ell+1/2$ & $2\ell$ & $\mu_1=1/\mu_2\in\mathbb{R}^{+}_*$ & Unstable \\ \hline
$p^-_r$ & $\ell-1/2<\eta_T^-<\eta_T^+=\ell$ & $2\ell$ & $\mu_1=\mu_2=1$ & Unstable\\ \hline $p^+_r$ &
$\ell=\eta_T^-<\eta_T^+<\ell+1/2$ & $2\ell$  & $\mu_1=\mu_2=1$ & Unstable \\ \hline $p^*_r$ &
$\ell=\eta_T^-=\eta_T^+$ & $2\ell$ & $\mu_1=\mu_2=1$ & Stable \\ \hline
\end{tabular}
\end{center} 

\begin{center}
\begin{tabular}{| l | c | c | l | l |}
\hline
\multicolumn{5}{|c|}{Case 2: $\ell<\eta_T^-\leq\eta_T^+<\ell+1$} \\ \hline
& $\eta_T$ & $i_T$ & Eigenvalues & Stability  \\ \hline
$h$ & $\ell<\eta_T^-<\ell+1/2<\eta_T^+<\ell+1$ & $2\ell+1$ & $\mu_1=1/\mu_2\in\mathbb{R}^{-}_*$ & Unstable \\ \hline
$p^-$ & $\ell<\eta_T^-<\eta_T^+=\ell+1/2$ & $2\ell+1$ & $\mu_1=\mu_2=-1$ & Unstable \\ \hline
$p^+$ & $\ell+1/2=\eta_T^-<\eta_T^+<\ell+1$ & $2\ell+1$ & $\mu_1=\mu_2=-1$ & Unstable \\ \hline
$p^*$ & $\ell+1/2=\eta_T^-=\eta_T^+$ & $2\ell+1$ & $\mu_1=\mu_2=-1$ & Stable \\ \hline
$e^-$ & $\ell<\eta_T^-\leq\eta_T^+<\ell+1/2$ & $2\ell+1$ & $\mu_1=\overbar{\mu_2}\in\mathbb{S}^1_*$ & Strong Stable \\ \hline
$e^+$ & $\ell+1/2<\eta_T^-\leq\eta_T^+<\ell+1$ & $2\ell+1$ & $\mu_1=\overbar{\mu_2}\in\mathbb{S}^1_*$  & Strong Stable \\ \hline
\end{tabular}
\end{center}
\end{theorem}

\begin{remark}\label{remdefres}
\rm In the $T$-resonant cases, the value of the Conley-Zehnder index $i_T = 2\ell$ has to be meant as a definition. As already remarked at the end of Section \ref{sec2.4}, this definition seems to be very natural from the point of view of the winding number $\eta_T$ of the solutions (cf. \cite[Rem. 2]{GiMa}), from at least two points of view. First, the $T$-resonant cases $(1,p^-_r)$, $(1,p^+_r)$ and $(1,p^*_r)$
turn out to be completely analogous to the $2T$-resonant ones $(2,p^-)$, $(2,p^+)$ and $(2,p^*)$. Second, all the cases with
$i_T = 2\ell$ (that is, all the cases in the first table) behave in the same way when giving integer \emph{strict} lower/upper bounds for the winding number $\eta_T$: indeed, in all the situations the best available information is that $\ell - 1 < \eta_T^- \leq \eta^+_T < \ell +1$.
This will be essential in the application of the Poincar\'e-Birkhoff fixed point theorem described in Section \ref{sec4}. For a comparison between our definition of $i_T$ with the classical notion of Morse index in the case of a scalar equation, see Section \ref{secMorse}.
\end{remark}

\begin{proof}
At first, we recall that the fact that the above cases are the only possible ones is precisely the content of Proposition \ref{prop2.1}.

Now, we focus on the relation between the winding number $\eta_T$ and the eigenvalues of $M(T)$. In view of the first property stated in Lemma \ref{le2.iii}, a real positive eigenvalue (resp., a real negative eigenvalue) of $M(T)$ exists if, and only if, the graph of the function $\eta_T$ intersects the horizontal line $\{(x,\ell) \}_{x \in \mathbb{R}}$ (resp., $\{(x,\ell + 1/2) \}_{x \in \mathbb{R}}$). Moreover, by Proposition \ref{prop2.1} nontrivial $T$-periodic solutions exist in the cases $(1,p^-_r)$, $(1,p^+_r)$ and $(1,p^*_r)$ and do not exist otherwise. Hence, it must be $\mu_1 = \mu_2 = 1$ in the cases $(1,p^-_r)$, $(1,p^+_r)$ and $(1,p^*_r)$
and $\mu_1 = 1/\mu_2 \in \mathbb{R}^+_*$ in the case $(1,h)$. The argument for the cases $(2,h)$, $(2,p^-)$, $(2,p^+)$ and $(2,p^*)$ is analogous, using again Proposition \ref{prop2.1}. Finally, in the cases $(2,e^-)$ and $(2,e^+)$ the eigenvalues must belong to $\mathbb{S}^1_*$ since the function $\eta_T$ does not take values in $\mathbb{Z}/2$ and thus real eigenvalues of $M(T)$ do not exist. 

Next, the relation between eigenvalues and stability is an almost direct consequence of Corollary \ref{corstab}. 
The only thing to be observed is that in the parabolic cases $(1,p^-_r)$ and $(1,p^+_r)$ the matrix $M(T)$ is not the identity 
(since otherwise $\eta_T$ would be constant) while $M(T) = I$ in the case $(1,p^*_r)$ (since, by (W2) of Lemma \ref{le2.iii} all the solutions of \ref{prop2.1} are $T$-periodic). An analogous argument is valid in the case $(2,p^-)$, $(2,p^+)$ and $(2,p^*)$.

Finally, the relation between the winding number $\eta_T$ and the Conley-Zehnder index $i_T$ is a consequence of \cite[Lemma 4]{MRZ} (see also \cite{GiMa}). We stress once more that the equality $i_T = 2\ell$ in the resonant cases $(1,p^-_r)$, $(1,p^+_r)$ and $(1,p^*_r)$ has to be intended as a definition. 
\end{proof}

\subsubsection{The Morse index of scalar second order equations}\label{secMorse}

In order to better understand the definition of the Conley-Zehnder index in the resonant case, we deal here with the scalar second order equation
\begin{equation}\label{eqhill}
u'' + q(t) u = 0, \qquad u \in \mathbb{R},
\end{equation}
where $q: \mathbb{R} \to \mathbb{R}$ is a continuous and $T$-periodic function. As well known, the above equation can be written as a linear planar Hamiltonian system like \eqref{2.1} by setting
$$
u' = v, \qquad v' = -q(t) u,
$$
that is, $z = (u,v) \in \mathbb{R}^2$ and $S(t) = \begin{pmatrix}
q(t) &\; 0 \\
0 &\; 1 \\
\end{pmatrix}$. Accordingly, we enter in the setting described in the previous sections and we can define the winding number $\eta_T$ of the solutions as well as the Conley-Zehnder index $i_T$. 

On the other hand, the notion of Morse index for an equation like \eqref{eqhill} is also often considered. To introduce it, we need to embed \eqref{eqhill} into the one-parameter family of equations 
\begin{equation}\label{eqhillauto}
u'' + (\lambda + q(t)) u = 0,
\end{equation}
where $\lambda$ has the role of eigenvalue. The spectral theory for the corresponding $T$-periodic problem is well known: precisely (see, for instance, \cite{MW}) the eigenvalues of \eqref{eqhillauto} can be written as
$$
\lambda_0 < \lambda_1 \leq \lambda_2 < \lambda_3 \leq \lambda_4 < \cdots < \lambda_{2\ell-1} \leq \lambda_{2\ell} \to +\infty,
$$
where each eigenvalue is meant to be simple (that is, that the associated eigenspace has dimension $1$) with corresponding eigenfunction having winding number equal to $\ell$.
With this in mind, the \emph{Morse index} $\mathfrak{m}_T$ of equation \eqref{eqhill} is defined as the number of (strictly) negative eigenvalues, that is
(with the convention $\lambda_{-1} = -\infty$)
$$
\mathfrak{m}_T = j \Longleftrightarrow \lambda_{j-1} < 0 \leq \lambda_{j}\,.
$$
Moreover, the \emph{augmented Morse index} $\mathfrak{m}^+_T$ is defined as the number of non-positive eigenvalues, that is
$$
\mathfrak{m}^+_T = j \Longleftrightarrow \lambda_{j-1} \leq 0 < \lambda_{j}\,.
$$
Notice that $\mathfrak{m}_T = \mathfrak{m}_T^+$ whenever equation \eqref{eqhill} is non-resonant (that is, $\lambda_j \neq 0$ for every $j$);
on the other hand $\mathfrak{m}^+_T-\mathfrak{m}_T$ is the dimension of the associated eigenspace when
$\lambda_j = 0$ for some $j$ (recall that such a dimension can be one or two, depending on whether one or two null eigenvalues exist).

Using the rotational characterization of the eigenvalues of \eqref{eqhillauto} (as presented, for instance, in \cite{GaZh00}, and based on the relation between rotation number and winding number, see Corollary \ref{corrot}), it can be proved that
$$
\mathfrak{m}_T = \mathfrak{m}^+_T = i_T = \left\{
\begin{array}{ll}
2\ell\;\;&\hbox{if}\;\;\lambda_{2\ell-1} < 0 < \lambda_{2\ell}\,,\\[2pt]
2\ell+1\;\;&\hbox{if}\;\;\lambda_{2\ell} < 0 < \lambda_{2\ell+1}\,,
\end{array} 
\right. 	
$$
see for instance \cite[Th. 3.2]{MaReTo14}. Hence, the Morse index coincides with (the augmented Morse index and with) the Conley-Zehnder index in the non-resonant case (this equality is also known as ``index theorem'' and actually holds true in any dimension, see \cite[Sec. 3.4.1]{Abb-book}).
On the other hand, in the resonant cases we have
$$
\left\{
\begin{array}{ll}
2\ell - 1 = \mathfrak{m}_T < 2\ell = \mathfrak{m}_T^+ = i_T\;\;&\hbox{if}\;\;\lambda_{2\ell-1} = 0 < \lambda_{2\ell}\,,\\[2pt]
2\ell  = \mathfrak{m}_T = i_T < 2\ell +1 = \mathfrak{m}_T^+\;\;&\hbox{if}\;\;\lambda_{2\ell-1} < 0 = \lambda_{2\ell}\,,\\[2pt]
2\ell - 1 = \mathfrak{m}_T < 2\ell = i_T < 2\ell + 1= \mathfrak{m}_T^+\quad\;\;&\hbox{if}\;\; \lambda_{2\ell-1} = 0 = \lambda_{2\ell}\,.
\end{array} 
\right. 	
$$
Thus, it always holds that
 $$
 \mathfrak{m}_T \leq i_T \leq \mathfrak{m}^+_T
 $$
but at least one of the inequality is strict when equation \eqref{eqhill} is resonant. See also Figure \ref{Fig2} for a graphical explanation.
\begin{figure}[h!]
\includegraphics[scale=0.44]{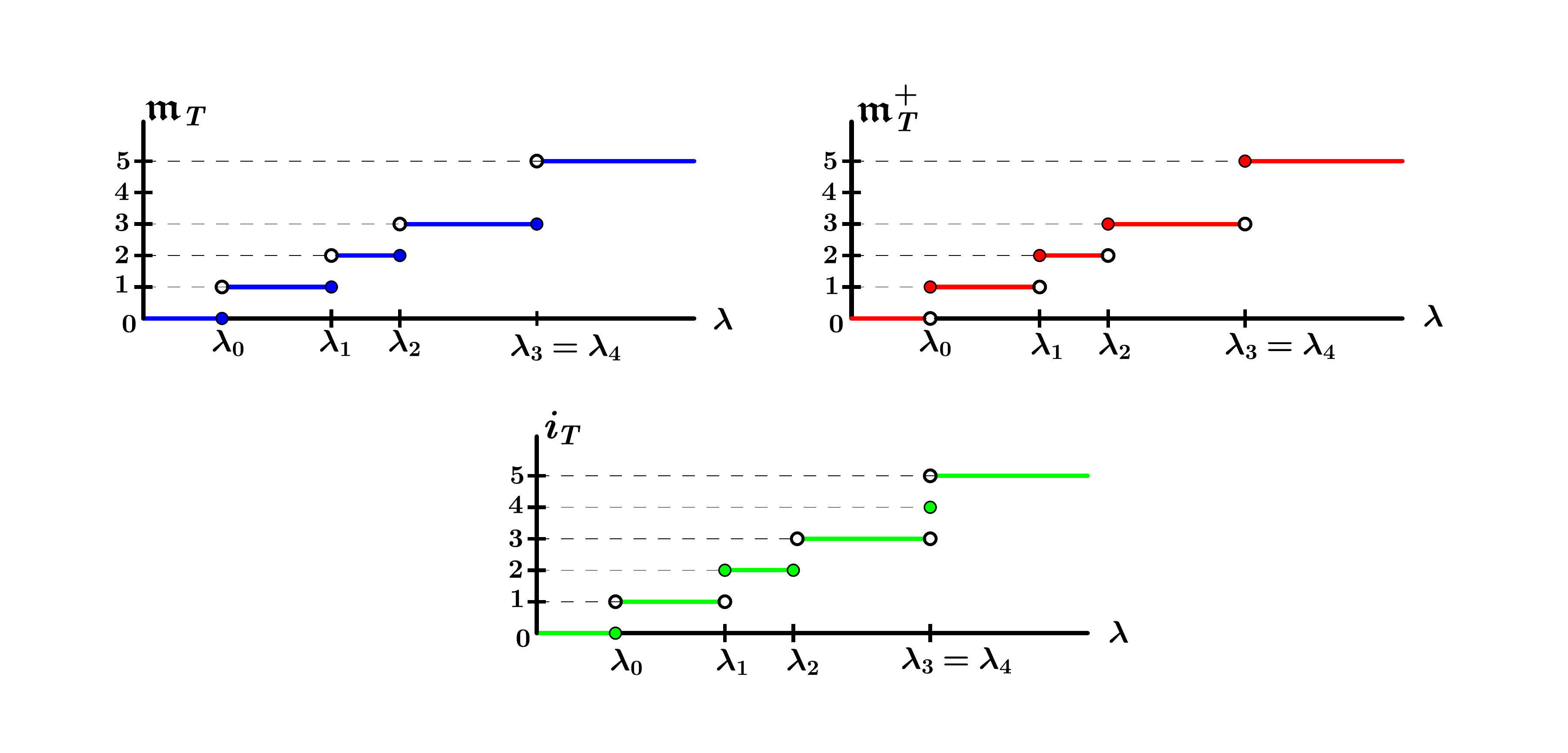}
\caption{\small For a continuous and $T$-periodic potential $q(t)$, the figure represents the graphs, as functions of $\lambda \in \mathbb{R}$, of the Morse index (in blue), of the augmented Morse index (in red) and of the Conley-Zehnder index (in green) of the parameter dependent equation $u'' + (\lambda + q(t)) u = 0$. On the $\lambda$-axis, we find the eigenvalues of the linear operator $u \mapsto -u'' - q(t)u$; notice that the Morse index (resp., the augmented Morse index) of $u'' + (\lambda + q(t)) u = 0$ is then the number of eigenvalues less than $\lambda$ (resp., less or equal to $\lambda$). The Conley-Zehnder index, on the other hand, is defined from the rotational properties of the solutions, as in this section. To illustrate the different behavior of the three indices, we have assumed that the eigenvalues $\lambda_1$ and $\lambda_2$ are simple and that the eigenvalue $\lambda_3 = \lambda_4$ is double (recall that the principal eigenvalue $\lambda_0$ is always simple). 
It is clear that the Conley-Zehnder index coincides with the Morse index or with the augmented Morse index when $\lambda$ is a simple eigenvalue (precisely, it coincides with the augmented Morse index if $\lambda$ is the left extreme of an hyperbolicity interval and with the Morse index if $\lambda$ is the right extreme) while it is different from both of them if $\lambda$ is a double eigenvalue. Notice that the Conley-Zehnder index is neither lower semi-continuous nor upper semi-continuous (while, as well known, the Morse index is lower semi-continuous and the augmented Morse index is upper semi-continuous).}
\label{Fig2}
\end{figure}

As already mentioned, our definition of $i_T$ seems to be very natural from the point of view of the rotation of solutions. Other definitions of $i_T$ in the resonant cases could be given, being more appropriate from other points of view: for instance, in \cite{Abb} the definition is made in such a way that the equality $\mathfrak{m}_T = i_T$ holds true also in the resonant case, see \cite[Sec. 1.3.7 and 3.4.1]{Abb-book}.

\section{Asymptotic indices and iteration formula}\label{sec3}

Motivated by the applications to the existence of subharmonic solutions discussed in Section \ref{sec4}, in this section 
we consider the linear Hamiltonian system \eqref{2.1} on the time interval $[0,kT]$, with $k \geq 1$ an integer number. 
Precisely, in Section \ref{sec3.1} we first provide some iteration formulas for the Conley-Zehnder index $i_{kT}$ of system \eqref{2.1} at time $kT$. Based on this, following \cite{Abb} in Section \ref{sec3.2} we recall the definition of mean Conley-Zehnder index for system \eqref{2.1} and we then compare this notion with the more classical one, of dynamical nature, of rotation number (that is, the mean winding number for the solutions to system \eqref{2.1}).

\subsection{Iteration formulas for the Conley-Zehnder index}\label{sec3.1}

Let $k \geq 1$ be an integer number and let $i_{kT}$ be the Conley-Zehnder index of system \eqref{2.1} at time $kT$, that is, the index for the sympectic path $M(t)$ on the time $[0,kT]$, with $M(t)$ the fundamental matrix of \eqref{2.1}.
\par
We start by considering the cases when system \eqref{2.1} is, at time $T$, hyperbolic or parabolic.  
In this situation, a sharp formula for $i_{kT}$ can be given, by simply taking advantage of the relation between Conley-Zehnder index, winding number and eigenvalues of the monodromy matrix stated in Theorem \ref{th2.2}. In order to give the complete prove, we need the following simple preliminary lemma, where the function $\eta_{kT}(\omega)$ is defined, similarly as in \eqref{2.9}, as the winding number on the interval $[0,kT]$ of the solution starting, at $t = 0$, at $e^{-i\omega}$.

\begin{lemma}
\label{le3.1}
For every integer $k \geq 1$, the next two properties are satisfied.
\item[($W4$)] If $\eta_T(\o)=\ell$ for some $\ell\in\Z/2$, then $\eta_{kT}(\o)=k\ell$,
\item[($W5$)] Given $\eta_T^{\pm}$ as defined in \eqref{minmax}, it holds that
\[
k\eta_T^-\leq \eta_{kT}(\o)\leq k\eta_T^+, \quad \text{ for every } \omega \in \mathbb{R}.
\] 
\end{lemma}
\begin{proof}
To prove (W4), we first observe that, as a consequence of the $T$-periodicity in $t$ of the differential equation \eqref{2.11}, it holds that
\begin{equation}\label{iterationtheta}
\theta(t+nT;\o) = \theta(t;\theta(nT;\omega)), \quad \text{ for any } t,\o \in \mathbb{R}, \, n \in \mathbb{Z}.
\end{equation}
Then, choosing $t = (k-1)T$ and $n=1$, we find
$$
\theta(kT;\omega) - \omega = \theta((k-1)T;\theta(T;\omega)) - \omega.
$$
Since $\theta(T;\omega) - \omega = 2\pi \ell \in \pi \mathbb{Z}$, formula \eqref{thetaomega} yields
$$
\theta(kT;\omega) - \omega = \theta((k-1)T;\omega) - \omega + 2\pi\ell.
$$
Iterating the argument, one finally finds
$$
\theta(kT;\omega) - \omega = \theta(0;\omega) - \omega + k2\pi\ell = k (\theta(T,\omega) - \omega). 
$$
This means that $\eta_{kT}(\omega) = k \eta_{T}(\omega) = k \ell$, as desired. 

To prove (W5), we use again \eqref{iterationtheta} to find
\begin{align*}
2\pi\eta_{2T}(\omega) & =\theta(2T;\omega) - \omega = \theta(T;\theta(T;\omega))-\theta(T;\omega) + \theta(T;\omega) - \omega \\
& = 2\pi (\eta_T(\theta(T;\omega)) +\eta_T(\omega))
\end{align*}
and hence
$$
2\eta^-_T \leq \eta_{2T}(\omega) \leq 2 \eta^+_T.
$$
Iterating the argument, the conclusion is obtained for every $k \geq 2$.
\end{proof}

We are now in a position to state the theorem for the iteration formula in the hyperbolic and parabolic cases.
The last column of each of the next tables provides the classification of the monodromy matrix $M(kT)$, according to the same labels as in Theorem \ref{th2.2}. To make a couple of examples, in the case $(1,p^-_r)$ the monodromy matrix $M(kT)$ is still of type $p^-_r$, meaning that it is parabolic, and $k\ell - 1/2 < \eta^-_{kT} < \eta^+_{kT} = k \ell$. In the case $(2,p^+)$, instead, the matrix $M(kT)$ is of type
$p^+$ when $k$ is odd, that is, is parabolic and 
$$
\mathbb{Z} + \tfrac{1}{2} \ni k\ell + (k-1)/2 + 1/2 = \eta^-_{kT} < \eta^+_{kT} <k\ell + (k-1)/2 + 1,
$$
while is of the type $p^+_r$ when $k$ is even that is, is parabolic and 
$$
\mathbb{Z} \ni k\ell + k/2 = \eta^-_{kT} < \eta^+_{kT} < k\ell + k/2 + 1/2.
$$
Notice that in this case the system is $kT$-resonant.

\begin{proposition}\label{itehp}
Let us assume that system \eqref{2.1} is either hyperbolic or parabolic. Then, with reference to the tables of Theorem \ref{th2.2}, for every integer $k \geq 1$ the following hold true: 
\begin{center}
\begin{tabular}{| l | c | c | l |}
\hline
\multicolumn{4}{|c|}{Case 1} \\ \hline  & $i_T$ & $i_{kT}$ & $M(kT)$-label  \\ \hline
$h$  & $2\ell$ & $2k\ell$ & $h\quad\;\; k\in\mathbb{Z}$ \\ \hline
$p^-_r$ & $2\ell$ & $2k\ell$ & $p^-_r\quad k\in\mathbb{Z}$ \\ \hline $p^+_r$ &
$2\ell$ & $2k\ell$ & $p^+_r\quad k\in\mathbb{Z}$ \\ \hline $p^*_r$ &
$2\ell$ & $2k\ell$ & $p^*_r\quad\, k\in\mathbb{Z}$ \\ \hline
\end{tabular}
\end{center} 

\begin{center}
\begin{tabular}{| l | c | c | l |}
\hline
\multicolumn{4}{|c|}{Case 2} \\ \hline 
& $i_T$ & $i_{kT}$ & $M(kT)$-label  \\ \hline
\multirow{2}{*}{$h$} & \multirow{2}{*}{$2\ell+1$} & 
$2(k\ell+k/2)\qquad\quad\;\, k\in 2\Z$ & \multirow{2}{*}{$h\quad\;\; k\in\mathbb{Z}$} \\ \cline{3-3} 
&  &$2(k\ell+(k-1)/2)+1\quad k\in 2\Z+1$ 
&
\\ \hline
\multirow{2}{*}{$p^-$} & \multirow{2}{*}{$2\ell+1$} & 
$2(k\ell+k/2)\qquad\quad\;\, k\in 2\Z$ &
$p^-_r\quad k\in 2\Z$\\ \cline{3-4}
&  &$2(k\ell+(k-1)/2)+1\quad k\in 2\Z+1$ 
&
$p^-\quad k\in 2\Z+1$
\\ \hline
\multirow{2}{*}{$p^+$} & \multirow{2}{*}{$2\ell+1$} & 
$2(k\ell+k/2)\qquad\quad\;\, k\in 2\Z$ &
$p^+_r\quad k\in 2\Z$\\ \cline{3-4}
&  & $2(k\ell+(k-1)/2)+1\quad k\in 2\Z+1$ 
&
$p^+\quad k\in 2\Z+1$
\\ \hline
\multirow{2}{*}{$p^*$} & \multirow{2}{*}{$2\ell+1$} & 
$2(k\ell+k/2)\qquad\quad\;\, k\in 2\Z$ &
$p^*_r\quad\, k\in 2\Z$\\ \cline{3-4}
&  &$2(k\ell+(k-1)/2)+1\quad\, k\in 2\Z+1$ 
&
$p^*\quad\, k\in 2\Z+1$
\\ \hline
\end{tabular}
\end{center} 
\end{proposition}
\begin{proof}
The relation between the columns for $i_T$ and $i_{kT}$ holds as an easy consequence of (W4) in Lemma \ref{le3.1}.
Indeed, by Theorem \ref{th2.2}, in all the cases of the first table the graph of the function $\eta_T$ intersects the horizontal line
$\{(x,\ell) \}_{x \in \mathbb{R}}$ and no-other horizontal lines at semi-integer ordinate. Then, by (W4), the graph of the function $\eta_{kT}$ intersects the horizontal line
$\{(x,k\ell) \}_{x \in \mathbb{R}}$. Using again Theorem \ref{th2.2} (at time $kT$) this implies $i_{kT} = k\ell$. The argument for the cases of the second table is similar, with the only difference that, since the graph of the function $\eta_T$ intersects the horizontal line
$\{(x,\ell+1/2) \}_{x \in \mathbb{R}}$, the value of $i_{kT}$ is even or odd\footnote{Notice that, in any case, $i_{kT} = k (2\ell+1)$, which however is written as
$2(k\ell + k/2)$ or $2 (k\ell + (k-1)/2) + 1$ to emphasize that $i_{kT}$ is even or odd depending on whether $k$ is even or odd.}
depending on whether $k$ is even or odd.

It remains to check that the labels for $M(kT)$ are the ones indicated in the last column of each tables. 
To this end, we first observe that $M(kT)=M(T)^k$ and, thus, if $\mu_1, \mu_2$ are the eigenvalues of $M(T)$, then the eigenvalues of $M(kT)$ are $\mu_1^k$ and $\mu_2^k$.  This immediately implies that $M(kT)$ is hyperbolic (resp., parabolic) if $M(T)$ is hyperbolic
(resp., parabolic). The fact that the superscripts $\{-,+,*\}$ are the same for $M(T)$ and $M(kT)$, instead, is an easy consequence of (W5) of Lemma \ref{le3.1} (incidentally, notice that the preservation of the superscript $*$ also follows from a stability argument). Notice that the parabolic cases of the second tables are $kT$-resonant when $k$ is even, thus inheriting the subscript $r$. 
\end{proof}

We now deal with the elliptic cases, namely, the cases $(2,e^-)$ and $(2,e^+)$ of Theorem \eqref{th2.2}.
In this situation, the function $\eta_T$ does not take semi-integer values and the previous winding number argument cannot be used.
However, the following proposition can be given.

\begin{proposition}\label{iteell}
Let us assume that system \eqref{2.1} is elliptic, with $i_{T} = 2\ell + 1$. Then, it holds that
$$
2k\ell +1 \leq i_{kT} \leq 2k\ell +k   \quad \mbox{ or } \quad 2k\ell + k \leq i_{kT} \leq 2k\ell + 2k-1, 
$$
depending on whether system \eqref{2.1} is of type $(2,e^-)$ or $(2,e^+)$.
Moreover, denoting by $e^{\pm i \varphi}$, with $\varphi \in (0,\pi)$, the eigenvalues of $M(T)$, the following holds true:
\begin{itemize}
\item if $\varphi$ is not commensurable with $\pi$, then $M(kT)$ is elliptic for every $k \geq 1$, and, thus, $i_{kT}$ is odd, 
\item if $\varphi$ is commensurable with $\pi$, say $\varphi = \pi p/q$ with $p,q$ coprime integers with $q \geq 2$ and 
$1 \leq p \leq (q-1)$, then
 $$
 \left\{
\begin{array}{ll}
$M(kT)$ \text{ is elliptic } \;\;&\hbox{if}\;\; k\notin q \Z,\\[2pt]
$M(kT)$ \text{ is parabolic stable (with eigenvalues $1$) }\;\;&\hbox{if}\;\;p \in 2\mathbb{N}, \, k \in q\Z,\\[2pt]
$M(kT)$ \text{ is parabolic stable (with eigenvalues $1$) }\;\;&\hbox{if}\;\; p \in 2\mathbb{N} + 1, \, k \in q2\Z,\\[2pt]
$M(kT)$ \text{ is parabolic stable (with eigenvalues $-1$) }\;\;&\hbox{if}\;\;p \in 2\mathbb{N} + 1, \, k \in q(2\Z+1),
\end{array} 
\right. 	
$$
and, thus, $i_{kT}$ is odd or even accordingly.
\end{itemize}
\end{proposition}

\begin{remark}\label{remsecond}
\rm For further convenience, we observe that the matrix $M(2T)$ is either elliptic or parabolic stable with eigenvalue $-1$.
Indeed, its eigenvalues are $e^{\pm 2 i \varphi}$ with $\varphi \in (0,\pi)$.
\end{remark}

\begin{proof}
The second part of the statement is a direct consequence of the fact, already observed in the proof of Proposition \ref{itehp}, that the eigenvalues of $M(kT)$ are $e^{\pm k \varphi}$. We thus focus on the Conley-Zehnder index estimate.
By (W5) of Lemma \ref{le3.1}, and assuming to fix the ideas that system \eqref{2.1} is of type $(2,e^-)$, it holds that
$$
k\ell < \eta_{kT}(\omega) <k\ell + k/2, \quad \mbox{ for every } \omega \in \mathbb{R}. 
$$
By Theorem \ref{th2.2} at time $kT$, the extreme possibilities for the winding number $\eta_{kT}$ which are compatible with the above estimate are the following.
On one hand, it could be 
$$
k\ell < \eta_{kT}(\omega) < k\ell + 1/2;
$$
in this case, system \eqref{2.1} is, at time $kT$, of type $(2,e^-)$ and so $i_{kT} = 2k\ell + 1$. On the other hand, 
it could be
$$
k\ell + (k-2)/2 +1/2 <\eta_{kT}(\omega) < k\ell + k/2 = k\ell + (k-2)/2 + 1 \quad \text{ if } k \in 2\Z
$$
or
$$
k\ell + (k-1)/2 <\eta_{kT}(\omega) < k\ell + k/2 = k\ell + (k-1)/2 + 1/2 \quad \text{ if } k \in 2\Z;
$$
in the former case, system \eqref{2.1} is, at time $kT$, of type $(2,e^+)$ and so 
$$
i_{kT} = 2\left( k\ell + (k-2)/2 \right) + 1 = 2k\ell + k -1,
$$
while in the latter one system \eqref{2.1} is, at time $kT$, of type $(2,e^-)$ and so
$$
i_{kT} = 2\left( k\ell + (k-1)/2) \right) + 1= 2k\ell + k.
$$
Summing up, we have obtained $2k\ell +1 \leq i_{kT} \leq 2k\ell +k$, as desired. The argument for a system of type $(2,e^+)$ is analogous.
\end{proof}

To conclude this section, we observe that from the above analysis we can easily obtain the following ``stability via second iteration'' result, see for instance \cite[Th. 6]{DO18}.

\begin{corollary}
Let us assume that system \eqref{2.1} is $T$-nonresonant and $2T$-nonresonant. Then, it is stable if, and only if, the indices 
$i_T$ and $i_{2T}$ are both odd.
\end{corollary}

\begin{proof}
At first, we notice that the parabolic cases $(1,p^-_r)$, $(1,p^+_r)$, $(1,p^*-_r)$ are ruled out since they are $T$-resonant, while the parabolic cases $(2,p^-)$, $(2,p^+)$, $(2,p^*)$) are ruled out since, by Proposition \ref{itehp}, they are $2T$-resonant (since $k=2$ is even).
Thus, system \eqref{2.1} is stable if, and only if, it is of the type $(2,e^-)$ or $(2,e^+)$. 

So, let us assume that system \eqref{2.1} is stable. By the previous discussion, $i_T$ must be odd. 
Moreover, by Proposition \ref{iteell} and Remark \ref{remsecond} the matrix $M(2T)$ is either elliptic or parabolic with eigenvalues equal to $-1$. Hence, $i_{2T}$ is also odd.

Conversely, let us assume that $i_T$ and $i_{2T}$ are both odd. From the fact that $i_T$ is odd we infer that system \eqref{2.1} is of type $(2,h)$, $(2,e^-)$ or $(2,e^+)$. However, assuming that it is of type $(2,h)$ we would get, by Proposition \ref{itehp}, that $i_{2T} = 2 i_{T}$ and so that $i_{2T}$ is even. Since this cannot happen, system \eqref{2.1} is elliptic and so it is stable.
\end{proof}

For a thorough treatment of the iteration theory for the Conley-Zehnder index, we refer to \cite[Sec. 4]{Long}.

\subsection{The rotation number, the mean Conley-Zehnder index and their relation}\label{sec3.2}

Following the ideas of Moser \cite{Mo81}, the equation \eqref{2.11} for the angular variation of a solution can be meant as an autonomous differential equation on a torus and so we can define the \textit{rotation number} $\rho$ of \eqref{2.1} as the real number
$$
\rho = \lim_{k \to +\infty} \frac{\eta_{kT}(\omega)}{k},
$$
where, as in the previous section, $\eta_{kT}(\omega)$ is the winding number on the interval $[0,kT]$ of the solution starting, at $t = 0$, at $e^{-i\omega}$. The fact that the above limit exists is proved in many textbooks, see for instance, in \cite[Th. 3.1]{Ha}. A less known property which we will need in the following is the global inequality 
\begin{equation}\label{rotglobal}
\vert \eta_{kT}(\omega) - k \rho \vert < 1, \quad \text{ for every } \omega \in [0,2\pi),\, k \in \mathbb{Z},
\end{equation}
for which we refer to \cite[inequality (2.5)]{He79}. Incidentally, notice that the above inequality implies that the limit defining the rotation number exists and it is uniform with respect to $\omega$.

On the other hand, the \textit{mean Conley-Zehnder index} $m$ of \eqref{2.1} is defined as the real number
$$
m = \lim_{k \to +\infty} \frac{i_{kT}}{k},
$$
where $i_{kT}$ denotes the Conley-Zehnder index of \eqref{2.1} at time $kT$, that is, the index for the sympectic path $M(t)$ on the time $[0,kT]$, with $M(t)$ the fundamentral matrix of \eqref{2.1}. This definition corresponds (up to a normalization constant) to the one given in \cite[Def. 3.4]{Abb} (see also \cite[Ch. 1, Sec. 6, Def. 2]{Ek}), where it is called mean winding number. Incidentally, notice that in \cite{Abb} the definition is given only for non-resonant systems, but we can here extend it to the general case according to the definition of the Conley-Zehnder index for resonant systems given in Theorem \ref{th2.2} (see also Remark \ref{remdefres}). 

The above indices are immediately computed if system \eqref{2.1} is either hyperbolic or parabolic. Indeed, using Theorem \ref{th2.2} and (W4) of Lemma \ref{le3.1} to compute $\rho$ and Proposition \ref{itehp} to compute $m$, it is easily seen that
$$
i_T = 2\ell \Longrightarrow 2\rho = m = 2\ell
$$
and 
$$
i_T = 2\ell + 1 \Longrightarrow 2 \rho = m = 2\ell + 1.
$$
Hence, $m = 2\rho$ in the hyperbolic and parabolic cases.
\par
The main result of this section establishes that this equality holds true also in the elliptic case. Precisely, we have the following result.

\begin{theorem}\label{mequalrho}
With reference to the tables of Theorem \ref{th2.2}, the following hold true:
\begin{center}
\begin{tabular}{| l | c | c | c | c |}
\hline
\multicolumn{5}{|c|}{Case 1} \\ \hline &
$\eta_T$ & $\rho$ & $i_T$ & $m$ \\ \hline
$h$ & $\ell-1/2<\eta_T^{-}<\ell<\eta_T^+<\ell+1/2$ & $\ell$ & $2\ell$ & $2\ell$ \\ \hline
$p^-_r$ & $\ell-1/2<\eta_T^-<\eta_T^+=\ell$ & $\ell$ & $2\ell$ & $2\ell$\\ \hline $p^+_r$ &
$\ell=\eta_T^-<\eta_T^+<\ell+1/2$ & $\ell$  & $2\ell$ & $2\ell$ \\ \hline $p^*_r$ &
$\ell=\eta_T^-=\eta_T^+$ & $\ell$ & $2\ell$ & $2\ell$\\ \hline
\end{tabular}
\end{center} 

\begin{center}
\begin{tabular}{| l | c | c | c | c|}
\hline
\multicolumn{5}{|c|}{Case 2} \\ \hline
& $\eta_T$ & $\rho$ & $i_T$ & $m$ \\ \hline
$h$ & $\ell<\eta_T^-<\ell+1/2<\eta_T^+<\ell+1$ & $\ell+1/2$ & $2\ell+1$ & $2\ell+1$ \\ \hline
$p^-$ & $\ell<\eta_T^-<\eta_T^+=\ell+1/2$ & $\ell+1/2$ & $2\ell + 1$ & $2\ell + 1$ \\ \hline
$p^+$ & $\ell+1/2=\eta_T^-<\eta_T^+<\ell+1$ & $\ell+1/2$ & $2\ell + 1$ & $2\ell + 1$ \\ \hline
$p^*$ & $\ell+1/2=\eta_T^-=\eta_T^+$ & $\ell+1/2$ & $2\ell + 1$ & $2\ell + 1$ \\ \hline
$e^-$ & $\ell<\eta_T^-\leq\eta_T^+<\ell+1/2$ & $\ell+\tau$ & $2\ell + 1$ & $2(\ell+\tau)$ \\ \hline
$e^+$ & $\ell+1/2<\eta_T^-\leq\eta_T^+<\ell+1$ & $\ell+\tau$ & $2\ell + 1$ & $2(\ell+\tau)$ \\ \hline
\end{tabular}
\end{center}
where $\tau \in (0,1/2)$ in the case $(2,e^-)$ and $\tau \in (1/2,1)$ in the case $(2,e^+)$. In particular, for every linear Hamiltonian system \eqref{2.1} it holds that $m = 2\rho$.
\end{theorem}

\begin{proof}
As observed in the discussion preceding the theorem, it is enough to consider the elliptic case; thus, we assume $i_T = 2\ell + 1$. From now on, we will use several times the covering projection \eqref{defcovering}. Having recalled this notation, we start with some preliminary considerations.

At first, we observe that from the formulas \eqref{formulaMRZ1} and \eqref{formulaMRZ2} (at time $kT$) it must hold
\begin{equation}\label{stima1}
\vert \pi i_{kT} - \vartheta(kT) \vert < \pi.
\end{equation}
Since by definition $i_{kT} = \delta(kT+1)$ where $\delta$ is the function introduced in \eqref{2.17}, and $\vert \delta(kT) - \delta(kT+1) \vert < 1$ by \cite[Formula (1)]{Abb},
from \eqref{stima1} we obtain
\begin{equation}\label{stima2}
\vert \pi\delta(kT) - \vartheta(kT) \vert < 2\pi.
\end{equation}

Second, we claim that, for every integer $k \geq 1$, there exists $\omega_k \in [0,\pi)$ such that
\begin{equation}\label{claim1}
\eta_{kT}(\omega_k) = \frac{\vartheta(kT)}{2\pi}.
\end{equation}
Indeed, let us take $\omega_k \in [0,\pi)$ such that
$O(\vartheta(kT))e^{-i\omega_k}$ is an eigenvector of the symmetric and positive definite matrix $P(\tau(kT),\sigma(kT))$. Then,
$$
M(kT) e^{-i\omega_k} = \lambda_k e^{-i(\omega_k + \vartheta(kT))}, \quad \text{ for some } \lambda_k > 0.
$$
Since, on the other hand,
$$
M(kT) e^{-i\omega_k} = \lambda_k e^{-i\theta(kT;\omega_k)}=\lambda_k e^{-i(\omega_k + 2\pi \eta_{kT}(\omega_k))},
$$
we deduce that $2\pi\eta_{kT}(\omega_k) - \vartheta(kT)$ must be an integer multiple of $2\pi$.
However, Theorem \ref{th2.2} (at time $kT$) implies that
$$
\vert 2\eta_{kT}(\omega) - i_{kT} \vert < 1, \quad \mbox{ for every } \omega \in \mathbb{R};
$$
using this information together with \eqref{stima1}, we thus obtain 
$$
\vert 2\pi\eta_{kT}(\omega_k) - \vartheta(kT) \vert \leq \pi \vert 2\eta_{kT}(\omega_k) - i_{kT} \vert + 
\vert \pi i_{kT} - \vartheta(kT) \vert < 2\pi.
$$
Thus, it must be $2\pi\eta_{kT}(\omega_k) - \vartheta(kT) = 0$, which corresponds to \eqref{claim1}.

Finally, we recall that, as shown in \cite{Abb}, 
\begin{equation}\label{defmabbo}
m = \delta(T)
\end{equation}
and, by \cite[Prop. 3.1]{Abb}, it holds $\delta(kT) = k \delta(T)$. We are now in a position to conclude the proof. Let us write
\begin{align*}
\frac{m}{2}-\rho & = \frac{\delta(T)}{2} - \rho = \frac{\delta(kT)}{2k} - \rho = \frac{\vartheta(kT)}{2\pi k} - \rho + 
\frac{\pi\delta(kT)-\vartheta(kT)}{2\pi k} \\
& =  \frac{\eta_{kT}(\omega_k)}{k} - \rho + 
\frac{\pi\delta(kT)-\vartheta(kT)}{2\pi k}.
\end{align*}
Then, using \eqref{rotglobal} and \eqref{stima2}, we find
$$
\left\vert \frac{m}{2}-\rho \right\vert \leq \left\vert \frac{\eta_{kT}(\omega_k)}{k} - \rho \right\vert + \frac{\vert\pi\delta(kT)-\vartheta(kT)\vert}{2\pi k} \leq \frac{2}{k}.
$$
Letting $k \to +\infty$, we thus infer $2\rho = m$, as desired.

To conclude the proof, it remains to check that  the rotation number is of the form 
$\rho = \ell + \tau$ with $\tau \in (0,1/2)$ in the case $(2,e^-)$ and $\tau \in (1/2,1)$ in the case $(2,e^+)$. This is an easy consequence of (W5) of Lemma \ref{le3.1}
(even more, we can say that $\tau \in (\eta^-_T-\ell,\eta^+_T-\ell)$). 
\end{proof}

\begin{remark}
\rm
Notice that, from $m=2\rho$, relation \eqref{defmabbo} and the definition of the rotation function \eqref{defrotfunction} given in Section \ref{sec2.3}, it follows that, if $\rho = \ell + \tau$,
the eigenvalues of $M(T)$ are given by $e^{\pm 2\pi i \tau}$. This can be used in connection with Proposition \ref{iteell} to study the behavior of the iterates $M(kT)$.
\end{remark}

We conclude the section with two corollaries of Theorem \ref{mequalrho}. The first one, which follows directly from the tables of the theorem, is the following ``stability via rotation number'', cf.  \cite[Prop. 3.3]{Abb}.

\begin{corollary}
Let us assume that system \eqref{2.1} is $T$-nonresonant and $2T$-nonresonant. Then, it is stable if and only if 
$\rho \in \mathbb{R} \setminus \mathbb{Z}/2$.
\end{corollary}

The second corollary provides a relation between the rotation number and the winding number $\eta_{kT}$ of the solutions at time $kT$, with $k \geq 1$ an arbitrary integer, cf. \cite[Lemma 2.2]{WaQi21}.

\begin{corollary}\label{corrot}
For every integer $k \geq 1$ and for every integer $j$, the following hold true:
$$
\begin{array}{lll}
\rho & = \frac{j}{k}  \Longleftrightarrow j-1 < \eta_{kT}^- \leq j\leq \eta_{kT}^+ < j+1 \vspace{0.2cm}\\
\rho & < \frac{j}{k} \Longleftrightarrow \eta_{kT}\leq\eta_{kT}^+ < j \vspace{0.2cm} \\
\rho & > \frac{j}{k} \Longleftrightarrow j < \eta_{kT}^-\leq \eta_{kT}, 
\end{array}
$$
where $\eta_{kT}^-$ and $\eta_{kT}^+$ denote, respectively, the minimum and the maximum of the function $\eta_{kT}(\omega)$.
\end{corollary}

\begin{proof}
For $k = 1$, the result is a direct consequence of the tables in Theorem \ref{mequalrho}. For $k > 1$, we argue as follows. Let us define
$$
\rho_k = \lim_{n \to +\infty}\frac{\eta_{nkT}(\omega)}{n}.
$$
Clearly, $\rho_k$ can be thought as the rotation number of system \eqref{2.1}, using however $kT$ instead of $T$ as period.
Hence, from the relation between winding number and rotation number, we have that
$$
\begin{array}{lll}
\rho_k & = j  \Longleftrightarrow j-1 < \eta_{kT}^- \leq j \leq\eta_{kT}^+ < j+1 \vspace{0.2cm}\\
\rho_k & < j \Longleftrightarrow \eta_{kT}\leq\eta_{kT}^+ < j \vspace{0.2cm} \\
\rho_k & > j \Longleftrightarrow j < \eta_{kT}^-\leq \eta_{kT}.
\end{array}
$$
Since
$$
\rho_k = k \lim_{n' \to +\infty}\frac{\eta_{n'T}(\omega)}{n'} = k \rho,
$$
the conclusion follows.
\end{proof}

\section{Subharmonic solutions to nonlinear Hamiltonian systems}\label{sec4}

In this section, we state and prove our results dealing with the existence of subharmonic solutions to planar Hamiltonian systems.
More precisely, in Section \ref{sec4.1} we provide our main result (Theorem \ref{thmain}), together with some corollaries, and we sketch an application to a Lotka-Volterra system with periodic coefficients. Then, in Section \ref{sec4.2} we focus on planar Hamiltonian systems coming from semilinear second order equations of the type $u'' + f(t,u) = 0$ or quasilinear equations of the type
$(\varphi(u'))' + f(t,u) = 0$ with $u \mapsto -(\varphi(u'))'$ a Minkowksi-curvature type operator. 

\subsection{The main result}\label{sec4.1}

Let us consider the nonlinear planar Hamiltonian system
\begin{equation}\label{hsnon}
Jz' = \nabla_z H(t,z), \qquad z \in \mathbb{R}^2,
\end{equation}
where $J$ is the standard symplectic matrix (as defined in \eqref{Jstandard}) and the Hamiltonian function $H: \mathbb{R} \times \mathbb{R}^2 \to \mathbb{R}$ is continuous, $T$-periodic in the $t$-variable and differentiable in $z$ with
$\nabla_z H$ continuous with respect to $(t,z)$. Moreover, we always assume that the uniqueness for the solutions of the Cauchy problems associated with system \eqref{hsnon} is guaranteed (for instance, we can assume that $\nabla_z H$ is Lipschitz continuous with respect to $z$ uniformly in time, but any other condition ensuring the uniqueness could be used). Finally, we suppose that system \eqref{hsnon} can be linearized at zero and at infinity, meaning that the following two assumptions are satisfied:
\begin{itemize}[align=left]
\item[$(H_0)$] there exists $S_0 \in \mathcal{C}([0,T];\mathcal{M}_2(\mathbb{R}))$ such that
$$
\nabla_z H(t,z) = S_0(t) z  + o (\vert z \vert), \qquad z \to 0, \, \text{ uniformly in } t \in [0,T],
$$
that is, $\nabla_z H(t,0) \equiv 0$ and $H$ is twice differentiable with respect to $z$ at $z = 0$ with $\nabla_{zz}^2 H(t,0) = S_0(t)$,
\item[$(H_\infty)$] there exists $S_\infty \in \mathcal{C}([0,T];\mathcal{M}_2(\mathbb{R}))$ such that
$$
\nabla_z H(t,z) = S_\infty(t) z  + o (\vert z \vert), \qquad \vert z \vert \to +\infty, \, \text{ uniformly in } t \in [0,T],
$$
\end{itemize}
Notice that, since $H$ is $T$-periodic in time, the matrices $S_0$ and $S_\infty$ are $T$-periodic, as well. Hence, we are allowed to use
the index theory developed in Section \ref{sec2} and Section \ref{sec3} for the linear planar Hamiltonian systems
$Jz' = S_0(t)z$ and $Jz' = S_\infty(t)z$.

In this setting, we are interested in the existence of subharmonic solutions of system \eqref{hsnon}. Let us recall that by a \emph{subharmonic solution of order $k$}, with $k \geq 2$ an integer number, we mean a $kT$-periodic solution which is not $lT$-periodic for any $l = 1, \ldots, k-1$. Notice also that, whenever a subharmonic solution $z$ of order $k$ exists, the $k-1$ functions $t \mapsto z(t + l T)$ for $l = 1,\ldots,k-1$ turn out to be subharmonic solutions of order $k$, as well. For this reason, we say that two subharmonic solutions $z_1$, $z_2$ are non-equivalent if $z_1(t) \not\equiv z_2(t + lT)$ for every $l = 1,\ldots, k-1$. Finally, by convention, we agree to call subharmonic solution of order $1$ a $T$-periodic solution of \eqref{hsnon}.

Our main result is the following.

\begin{theorem}\label{thmain}
Let us suppose that assumptions $(H_0)$ and $(H_\infty)$ hold true and denote by $\rho_0$ and $\rho_\infty$ the rotation numbers
of the linear systems $Jz' = S_0(t)z$ and $Jz' = S_\infty(t)z$, respectively. Assume that
\begin{equation}\label{rhodifferent}
\rho_0 \neq \rho_\infty.
\end{equation}
Then, there exists an integer $k^* \geq 1$ such that system \eqref{hsnon} possesses subharmonic solutions of order $k$, for every $k \geq k^*$. Moreover, the number of non-equivalent subharmonic solutions of order $k$ goes to infinity as $k \to +\infty$.
\end{theorem}

\begin{remark}
\rm
As it will be clear from the proof, the following nodal characterization can be provided: there are two non-equivalent subharmonic solutions of order $k$ and winding number $j$ in the interval $[0,kT]$ whenever $k,j$ are coprime (with $j \neq 0$; by coprime here we mean that $\gcd(k,\vert j \vert ) = 1$) and
\begin{equation}\label{nodal}
\frac{j}{k} \in  \left(\min\{\rho_0,\rho_\infty\},\max\{\rho_0,\rho_\infty\}\right).
\end{equation}
In particular, let us notice that the choice $k = 1$ is admissible if $\vert \rho_\infty - \rho_0 \vert >1$: hence, in this case system \eqref{hsnon} possesses at least two $T$-periodic solutions. By a careful inspection of the proof, one also sees that it is possible to take $j = 0$ when $\rho_0$ and $\rho_\infty$ have opposite sign: in this case, for $k=1$ one obtains the existence of two $T$-periodic solutions with zero winding number, but it cannot be guaranteed that other choices of $k \geq 2$ lead to subharmonic solutions. 
\end{remark}

Theorem \ref{thmain} can be compared with \cite[Th. 1.1]{WaQi21}, where, however, a slightly different situation is considered. 

In view of Theorem \ref{mequalrho} giving the relation between rotation number and mean Conley-Zehnder index, we immediately see that 
assumption \eqref{rhodifferent} of Theorem \ref{thmain} can be equivalently written as $m_0 \neq m_\infty$.
Moreover, in view of the relation between rotation number and Conley-Zehnder index (at time $T$) we also easily obtain the following corollary of Theorem \ref{thmain}. Notice that this result provides a substantial generalization of \cite[Cor. 4.3]{Abb}, since, on one hand, we do not impose any non-resonance assumptions and, on the other hand, we can provide a much sharper abundance of subharmonic solutions.

\begin{corollary}\label{corcz}
Let us suppose that assumptions $(H_0)$ and $(H_\infty)$ hold true and denote by $i_0$ and $i_\infty$ the Conley-Zehnder indices, at time $T$, of the linear systems $Jz' = S_0(t)z$ and $Jz' = S_\infty(t)z$, respectively. 
Assume that
$$
i_0 \neq i_\infty.
$$
Then, there exists an integer $k^* \geq 1$ such that system \eqref{hsnon} possesses subharmonic solutions of order $k$, for every $k \geq k^*$. Moreover, the number of non-equivalent subharmonic solutions of order $k$ goes to infinity as $k \to +\infty$.
\end{corollary}

We now provide the proof of Theorem \ref{thmain}.

\begin{proof}[Proof of Theorem \ref{thmain}]
According to the strategy described in \cite{Bo11}, we are going to apply the Poincar\'e-Birkhoff fixed point theorem (in the version for non-invariant annuli) to the Poincar\'e map at time $kT$ of system \eqref{hsnon}, that is
$$
\mathbb{R}^2 \ni z_0  \mapsto \Phi_{kT}(z_0) := z(kT;z_0) \in \mathbb{R}^2,
$$
where $z(\cdot;z_0)$ is the (unique) solution of \eqref{hsnon} satisfying the initial condition $z(0;z_0) = z_0$. 
Notice that all the solutions are globally defined in $\mathbb{R}$, since, due to $(H_\infty)$, the Hamiltonian vector field $J\nabla_z H(t,z)$ grows at most linearly w.r.t. $z$. Hence, the map $\Phi_{kT}$ is well-defined and, as well known, it turns out to be an homeomorphism of the plane onto itself; moreover, by the Liouville theorem, it is area-preserving. We also observe that, in view of $(H_0)$, one has $z(t;z_0) \neq 0$ for every $t \in \mathbb{R}$ whenever $z_0 \neq 0$; hence, $\Phi_{kT}(z_0) = 0$ if, and only if, $z_0 = 0$ and we can define the winding number around the origin of any non-trivial solution, namely
$$
\textnormal{Rot}_k(z_0) = \frac{\theta(kT;z_0) - \theta(0;z_0)}{2\pi}
$$
where the continuous function $\theta(t;z_0)$ is a (clockwise) argument for the solution $z(t;z_0)$, that is
$z(t;z_0) = \vert z(t;z_0) \vert e^{-i\theta(t;z_0)}$.

By the Poincar\'e-Birkhoff theorem (see again \cite{Bo11} for the details) the existence of two non-equivalent $kT$-periodic solutions of \eqref{hsnon} having winding number $j$ (with $j$ a fixed integer number) on the interval $[0,kT]$ is guaranteed if there exist 
$\hat{r}_{k,j} > 0$ and $\check{r}_{k,j} > 0$, with $\hat{r}_{k,j} \neq \check{r}_{k,j}$ such that the following twist condition holds true:
\begin{itemize}
\item[(T1)] for every $z_0 \in \mathbb{R}^2$ with $\vert z_0 \vert = \hat{r}_{k,j}$ one has
$\textnormal{Rot}_k (z_0) > j,$
\item[(T2)] for every $z_0 \in \mathbb{R}^2$ with $\vert z_0 \vert = \check{r}_{k,j}$ one has
$\textnormal{Rot}_k (z_0) < j.$
\end{itemize}
Moreover, by standard arguments, the period $kT$ can be checked to be the minimal one (in the class of the integer multiples of $T$) when $k$ and $j$ are coprime numbers and $j \neq 0$; thus, in such a case the $kT$-periodic solutions found are guaranteed to be subharmonics of order $k$.

To prove that the above twist condition is satisfied, we use a comparison argument with the solutions of the linearizations of the nonlinear system \ref{hsnon} at zero and at infinity. Indeed, by \cite[Lemma 3]{MRZ} (at time $kT$), conditions (T1) and (T2) are satisfied if the solutions of the linear system $Jz' = S_0(t)z$ have winding number, on the time $[0,kT]$, greater than $j$, and the solutions of the linear system $Jz' = S_\infty(t)z$ have winding number, on the time $[0,kT]$, less than $j$, or viceversa (with $\hat{r}_{k,j} < \check{r}_{k,j}$ in the former case, and $\hat{r}_{k,j} > \check{r}_{k,j}$ in the latter one). 

From now on, to fix the ideas we suppose that $\rho_\infty<\rho_0 $. By Corollary \ref{corrot}, the solutions of the linear system $Jz' = S_0(t)z$ have winding number at time $kT$ greater than $j$ if, and only if,
$$
\rho_0 > \frac{j}{k}
$$
while the solutions of the linear system $Jz' = S_\infty(t)z$ have winding number at time $kT$ less than $j$ if, and only if,
$$
\rho_\infty < \frac{j}{k}.
$$
Summarizing, we have thus proved that if $j$ is a non-null integer satisfying condition \eqref{nodal} and such that $k$ and $\vert j \vert$ are coprime, then there exists two non-equivalent subharmonics of order $k$ of system \eqref{hsnon} having winding number equal to $j$.

To conclude the proof, we thus need to show that for every integer $r \geq 1$, there exists an integer $k^*(r) \geq 1$ such that
for every $k \geq k^*(r)$ there exist at least $r$ non-null integers $j_1,\ldots,j_r$, all coprime with $k$, satisfying condition \eqref{nodal}. 
The basic idea of this proof appears in \cite[Th. 2.3]{DIZ93} and is also sketched in \cite[Th. 5.5]{BoGa11} but we include it here for the sake of completeness, assuming to fix the ideas that $0\leq\rho_\infty<\rho_0 $ (so that $j_1,\ldots,j_r \geq 0$; the argument
in the other cases is easily adapted, cf. the discussion  before Corollary \ref{corzero}).

First, by the Prime Number theorem, for every $a,b \geq 0$ with $a,b$ there exists an integer $\widetilde{k} = \widetilde{k}(a,b) \geq 1$ such that for $k\geq \widetilde{k}$ 
\begin{equation}
\label{prime}
a k<m(k)<bk
\end{equation}
where $m(k)$ is a prime integer, cf. \cite[formula (2.10)]{DIZ93}. Now, to find $r$ integers coprime with $k$ satisfying condition \eqref{nodal}, we consider a partition of the real interval $[\rho_{\infty},\rho_0]$ in $2r$ intervals by taking $s_i\in[\rho_\infty,\rho_0]$, for $i\in\{0,1,\dots,2r\}$, such that 
\[
\rho_\infty=s_0<s_1<s_2<\cdots<s_{2r-1}<s_{2r}=\rho_0. 
\]
Then, by \eqref{prime}, there is an integer $\hat k = \hat k(s_0,\ldots,s_{2r}) \geq 1$ such that for all $k\geq \hat k$ and for all $i\in\{0,1,\dots,r-1\}$ the next inequalities
\begin{equation}
\label{prime2}
ks_{2i}<j_{i,1}(k)<ks_{2i+1}\quad\hbox{and}\quad ks_{2i+1}<j_{i,2}(k)<ks_{2i+2}
\end{equation} 
hold for $j_{i,1}(k)$ and $j_{i,2}(k)$ suitable prime integers. We prove that, for all $i\in\{0,1,\dots,r-1\}$. either $j_{i,1}(k)$ or $j_{i,2}(k)$ is coprime with $k$, provided  
that
\begin{equation}\label{defSigma}
k > \Sigma:= \max_{i=0,\ldots,r-1} \frac{1}{s_{2i}s_{2i+1}}.
\end{equation}
Indeed, assume by contradiction that $j_{i,1}(k)$ and $j_{i,2}(k)$ are prime factors of $k$. Then, for some integer $k_1 \geq 1$ we can write $k$ as
\begin{equation}
\label{prime3}
k=k_1j_{i,1}(k)j_{i,2}(k)\geq j_{i,1}(k)j_{i,2}(k).
\end{equation}
Hence, by \eqref{prime2} and \eqref{prime3}, $k\geq k^2s_{2i}s_{2i+1}$, contradicting \eqref{defSigma}.
Therefore, 
if we take
$$
k^* = \max\left\{\hat k, \lfloor \Sigma \rfloor + 1 \right\},
$$
for every integer $k \geq k^*$, we find for every $i\in\{0,1,\dots,r-1\}$ an integer $j_i = j_i(k)$ coprime with $k$ and such that
$$
ks_{2i}<j_i<ks_{2(i+1)}.
$$
Hence, we have found $r$ integers coprime with $k$ and satisfying condition \eqref{nodal}. This concludes the proof.
\end{proof}

We now observe that the nodal characterization of the subharmonic solutions provided by \eqref{nodal} can be made more explicit in the case when
$$
\left[\min\{\rho_0,\rho_\infty\},\max\{\rho_0,\rho_\infty\}\right] \cap \mathbb{Z} \neq \emptyset.
$$
To explain this claim, we focus, to fix the ideas, on the case $0 \leq \rho_\infty < \rho_0$ and we thus assume that, for some integer $\ell \geq 0$, it holds that 
$$
\rho_\infty \leq \ell < \rho_0
$$
(the case $\rho_\infty < \ell \leq \rho_0$ is analogous).  
Then, condition \eqref{nodal} is certainly satisfied if $j$ is a (positive) integer satisfying 
$$
k\ell < j < k\rho_0
$$
and it is thus quite natural to look for $j$ of the form $j = k \ell + j'$. Observing that
$$
\gcd(k,j) = 1 \;\Longleftrightarrow \; \gcd(k,j') = 1
$$
we are thus led to find the integers $j'$ coprime with $k$ and such that
\begin{equation}\label{condjprime}
0 < j' < k(\rho_0 - \ell).
\end{equation}
In particular, defining $k^*$ as the least integer greater or equal to $1$ such that $k^* (\rho_0 - \ell) \geq 2$, it is always possible to take $j' = 1$ and we thus find, for every integer $k \geq k^*$, at least two subharmonic solutions of order $k$ with winding number equal to 
$k\ell + 1$. More in general, the number of integers $j'$ satisfying condition \eqref{condjprime} is given by the Euler's function
$\varphi(\left \lfloor{ k(\rho_0 - \ell)}\right \rfloor)$ and, clearly, diverges as $k \to+\infty$. We can thus find subharmonics of order $k$ and winding number
$$
k\ell + 1 < k\ell + j'_2 < \cdots k \ell + j'_{r_k}
$$
where $r_k = \varphi(\left \lfloor{ k(\rho_0 - \ell)}\right \rfloor) \to +\infty$ as $k \to +\infty$.

The case when $\rho_0< \rho_\infty \leq 0$ is of course completely analogous (just considering rotations in the counterclockwise sense), 
as well as the cases $0 \leq \rho_0 < \rho_\infty$ and $\rho_\infty < \rho_0 \leq 0$. Finally, the case when $\rho_0$ and $\rho_\infty$ have different sign gives rise to an even richer multiplicity, since the above argument can be applied both on the interval $[\min\{\rho_0,\rho_\infty\},0]$ and $[0,\max\{\rho_0,\rho_\infty\}]$ leading to subharmonic solutions rotating clockwise and to subharmonic solutions rotating counterclockwise.  

To conclude this discussion, we state an explicit corollary dealing with a situation which is very useful in the applications.

\begin{corollary}\label{corzero}
Let us suppose that assumptions $(H_0)$ and $(H_\infty)$ hold true and denote by $\rho_0$ and $\rho_\infty$ the rotation numbers
of the linear systems $Jz' = S_0(t)z$ and $Jz' = S_\infty(t)z$, respectively. Assume that
\begin{equation}\label{rhodifferent2}
\rho_\infty = 0 \neq \rho_0.
\end{equation}
Then, there exists an integer $k^* \geq 1$ such that system \eqref{hsnon} possesses subharmonic solutions of order $k$, for every $k \geq k^*$. More precisely, for every $k \geq k^*$ system \eqref{hsnon} possesses at least two non-equivalent subharmonic solutions of order $k$ and winding number $\textnormal{sgn}(\rho_0) j'$, for every integer $j'$ coprime with $k$ and satisfying $1 \leq j' < k \vert \rho_0 \vert$.
\end{corollary} 

Notice that, by Theorem \ref{mequalrho}, assumption \eqref{rhodifferent2} can be equivalently written in terms of Conley-Zehnder indices as
$$
i_\infty = 0 \neq i_0.
$$
We conclude the section with the following observation, which can be useful 
to establish some variants of the results just exposed. As it is clear from the proof of Theorem \ref{thmain}, the assumptions at zero and at infinity (that is, $(H_0)$ and $(H_\infty)$ and the related information on $\rho_0$ and $\rho_\infty$) have the role of dictating, through a comparison with the associated linearized systems, the angular behavior of the ``small'' and ``large'' solutions of the nonlinear system \eqref{hsnon}, respectively. As a consequence, all the results presented remain true if one of such conditions, at zero or at infinity, is replaced by an alternative one which ensures the same angular behavior for the solutions.

In particular, in some concrete applications system \eqref{hsnon} cannot be linearized at infinity in the sense of condition $(H_\infty)$ but, in spite of this, it is still possible to prove that the solutions are globally defined and that the following rotational property holds true:
\begin{itemize}[align=left]
\item[$(H_\textrm{sub})$] for any integer $k \geq 1$, there exists a radius $R_k > 0$ such that,
for every $\vert z_0 \vert = R_k$,  $$|\textnormal{Rot}_k(z_0) | < 1.$$
\end{itemize}
This is precisely the rotational property provided by the condition $\rho_\infty = 0$ for systems which can be linearized at infinity; however, 
as discussed in \cite{Bo11}, it can be
ensured also in cases when system \eqref{hsnon}, even if non-linearizable at infinity, satisfies a suitable sublinearity condition in a region of the plane (see also \cite{FoTo19} for related results, in the context of weakly coupled higher-dimensional systems). 
Hence, for instance, a variant of Corollary \ref{corzero} can be given, by keeping assumption $\rho_0 > 0$ in order to ensure a positive winding number for the small solutions and replacing assumption $\rho_\infty = 0$ with any other one implying $(H_\textrm{sub})$. To better explain this observation, we present the following application.

\begin{example}
\rm 
As an example, we consider the following Lotka-Volterra system
\begin{equation}\label{lotv}
\left\{
\begin{array}{ll}
p' = p (a(t)-b(t)q) \\
q' = q (-c(t) + d(t)p),
\end{array}
\right.
\qquad (p,q) \in (\mathbb{R}^+)^2,
\end{equation}
where $a,b,c,d: \mathbb{R} \to \mathbb{R}$ are continuous and $T$-periodic functions, with 
\begin{equation}\label{coex}
\int_0^T a(t)\,dt > 0, \qquad \int_0^T c(t)\,dt > 0
\end{equation}
and $b(t)$ and $d(t)$ \emph{non-negative} and non-identically equal to zero. As well known (see \cite{DiZa92} or \cite[Appendix]{LMZ20}), condition \eqref{coex} turns out to be necessary and sufficient for the existence of a $T$-periodic solution $(p^*(t),q^*(t))$ of system \eqref{lotv}; through the change of variables
$$
x = \log \left( \frac{q}{q^*(t)}\right), \qquad y = \log \left( \frac{p}{p^*(t)}\right), 
$$
and denoting $\a(t)=b(t)q^*(t)$ and $\b(t)=d(t)p^*(t)$ we are thus led to the planar system
\begin{equation}\label{lotv2}
\left\{
\begin{array}{ll}
x' = \b(t)(e^y-1) \\
y' = \a(t)(1-e^x),
\end{array}
\right.
\end{equation}
which is of the form \eqref{hsnon} with 
$$
H(t,x,y) = \a(t)(e^x - x -1) + \b(t)(e^y - y +1).
$$
As already proved in previous papers (see, for instance, \cite{DiZa92} and \cite[Sec. 2]{LMZ20}), solutions of the above system are globally defined and, moreover, assumption $(H_\textrm{sub})$ holds true.

On the other hand, using the results of Sections \ref{sec2} and \ref{sec3} we can easily verify that, for the linearization at zero of system \eqref{lotv2}, that is
\begin{equation}\label{lotv3}
\left\{
\begin{array}{ll}
x' = \b(t)y \\
y' = -\a(t)x,
\end{array}
\right.
\end{equation}
it holds that $\rho_0 \neq 0$. Indeed, assume by contradiction that $\rho_0 = 0$; then 
by Theorem \ref{mequalrho} and Proposition \ref{prop2.1} there exists a non-trivial $T$-periodic solution
$(\bar x(t),\bar y(t))$ of \eqref{lotv3} with winding number $\eta_T=0$.
Writing this solution in clockwise polar coordinates
$(\bar x(t),\bar y(t)) = \bar\rho(t) e^{-i\bar\theta(t)},$ we have 
$$
\bar\theta'(t) = \frac{\a(t) \bar x(t)^2 + \b(t) \bar y(t)^2}{\bar x(t)^2 + \bar y(t)^2} \geq 0, \quad \mbox{ for every } t \in [0,T].
$$
Then, since $\eta_T=0$, it must be $\bar\theta(t) = \bar\theta(0)$ for every $t \in [0,T]$: that is, 
$(\bar x(t),\bar y(t))$ moves periodically on the half-line $\{\theta = \bar\theta(0)\}$. Therefore, $\bar x'(t)$ and $\bar y'(t)$ must have constant sign on $[0,T]$ and the only possibility for $(\bar x(t),\bar y(t))$ to be $T$-periodic is that it is constant. However, this can be excluded since $\a(t)$ and $\b(t)$ are not identically equal to zero. Incidentally, let us notice that this proves that $\rho_0 > 0$.

Hence, the above described variant of Corollary \ref{corzero} can be applied, thus ensuring that system \eqref{lotv2} admits subharmonic solutions of order $k$, for any $k \geq k^*$. In this way, we obtain a slight improvement of the result exposed in \cite[Sec. 5]{Bo11}, where the functions $b(t)$ and $d(t)$ are assumed to be, instead, strictly positive. 
For some sharp estimates of the minimal order $k^*$ of these subharmonic solutions, depending on the coefficient supports, we refer to the recent papers \cite{LM20,LMZ20,LMZ21}.
\end{example}

\subsection{Some corollaries for second order ODEs}\label{sec4.2}

In this section, we discuss possible applications of the results of Section \ref{sec4.1} to planar Hamiltonian systems coming from some classes second order scalar differential equations.

At first, we focus on semilinear equations of the type
\begin{equation}\label{eqsem}
u'' + f(t,u) = 0, \qquad u \in \mathbb{R},
\end{equation}
where $f: \mathbb{R} \times \mathbb{R} \to \mathbb{R}$ is a locally Lipschitz continuous function, $T$-periodic in the first variable
and satisfying the following assumption at infinity:
\begin{itemize}
\item[$(f_\infty)$] there exists a continuous function $q_\infty: [0,T] \to \mathbb{R}$ such that
$$
\lim_{\vert u \vert \to +\infty} \frac{f(t,u)}{u} = q_\infty(t) \, \text{ uniformly in } t \in [0,T],
$$
\end{itemize}
and we denote by $\rho_\infty$ the rotation number of the linear equation $u'' + q_\infty(t)u = 0$, interpreted as a planar Hamiltonian system as discussed in Section \ref{secMorse}.

Moreover, we adopt the following notation: if $\bar{u}(t)$ is a $T$-periodic solution of equation\eqref{eqsem} such that
$\partial_u f(t,\bar{u}(t))$ is well-defined, by $\rho_{\bar{u}}$ we mean the rotation number of the linear equation 
$u'' + \partial_u f(t,\bar{u}(t))u = 0$. 

Incidentally, let us observe that $\rho_{\bar{u}} \geq 0$ and $\rho_{\infty} \geq 0$. This is a consequence of the well-known fact that the 
the winding number of the solutions of linear equations of the type $u'' + q(t)u = 0$ is always greater than $-1/2$, since, due to $u' = v$, the $v$-axis can be crossed only in the clockwise sense. 

With this in mind, the following result follows in a straightforward manner from Theorem \ref{thmain} of Section \ref{sec4.1}. 

\begin{theorem}\label{thsem}
Let us suppose that assumption $(f_\infty)$ holds true. Moreover, assume that there exists a $T$-periodic solution $\bar{u}(t)$ of equation \eqref{eqsem} such that
$$
\rho_{\bar{u}} \neq \rho_{\infty}.
$$
Then, there exists an integer $k^* \geq 1$ such that equation \eqref{eqsem} possesses subharmonic solutions of order $k$, for every $k \geq k^*$. Moreover, the number of non-equivalent subharmonic solutions of order $k$ goes to infinity as $k \to +\infty$.
\end{theorem}

Notice that this result can be seen as a generalization of \cite[Th. 2.1]{BoOrZa14}, where the case $\rho_\infty = 0$ (and $\bar{u} \equiv 0$) is considered.

\begin{proof}[Proof of Theorem \ref{thsem}]
We first observe that a function $u(t)$ is a subharmonic solution of \eqref{eqsem} if, and only if, the function
$x(t) := u(t) - \bar{u}(t)$ is a subharmonic solution of the equation
\begin{equation}\label{eqsemaux}
x'' + f(t,x + \bar{u}(t)) - f(t,\bar{u}(t)) = 0.
\end{equation}
Accordingly, let us define, for $(x,y) \in \mathbb{R}^2$ and $t \in \mathbb{R}$, the Hamiltonian function
$$
H(t,x,y) = \frac{1}{2} y^2 + \int_0^x \left( f(t,\xi + \bar{u}(t)) - f(t,\bar{u}(t))\right) d\xi,
$$
giving rise to the planar Hamiltonian system
\begin{equation}\label{hsproof}
\left\{
\begin{array}{ll}
x' = y \\
y' = -f(t,x + \bar{u}(t)) + f(t,\bar{u}(t)).
\end{array}
\right.
\end{equation}
It is immediately checked that $H$ satisfies assumptions $(H_0)$ and $(H_\infty)$ of Section \ref{sec4.1} and that the linearized system at zero and at infinity are given, respectively, by
$$
\left\{
\begin{array}{ll}
x' = y \\
y' = - \partial_x f(t,\bar{u}(t))x,
\end{array}
\right. \qquad \mbox{ and } \qquad 
\left\{
\begin{array}{ll}
x' = y \\
y' = - q_\infty(t)x.
\end{array}
\right.
$$
The rotation numbers of these systems are nothing but the rotation number $\rho_{\bar{u}}$ and $\rho_{\infty}$ as defined before the statement of Theorem \ref{thsem}. Since these numbers are different, Theorem \ref{thmain} can be applied, yielding the existence of subharmonic solutions $(x(t),y(t))$ for the Hamiltonian system \eqref{hsproof}. Then, the function $x(t)$ is a subharmonic solution of equation 
\eqref{eqsemaux}; accordingly, $u(t) = x(t) + \bar{u}(t)$ 
is a subharmonic solution of equation \eqref{eqsem} and the conclusion follows. 
\end{proof}

Of course, on the lines of Corollary \ref{corcz}, the assumption of Theorem \ref{thsem} can equivalently be expressed in terms of Conley-Zehnder indices: indeed, if $i_{\bar{u}}$ and $i_{\infty}$ denote the Conley-Zehnder indices of the linearized systems, then 
$i_{\bar{u}} \neq i_{\infty}$ implies that $\rho_{\bar{u}} \neq \rho_\infty$. It is worth noticing, however, that
the analogous implication does not hold, in general, if the Conley-Zehnder index is replaced by the usual Morse index (or by the augmented Morse index). Indeed, as discussed in Section \ref{secMorse}, the definition of the Morse index is not natural, from the point of view of the rotation number, when the linear system is resonant (referring for instance to Figure \ref{Fig2}, we have that $1 = \mathfrak{m}_T(\lambda_1) \neq \mathfrak{m}_T(\lambda_2)$, while $\rho(\lambda) = 1$ for every $\lambda \in [\lambda_1,\lambda_2]$). 

As a a second example of applications, we focus on quasilinear second order equations of the type
\begin{equation}\label{eqqua}
(\varphi(u'))' + f(t,u) = 0, \qquad u \in \mathbb{R},
\end{equation}
where again $f: \mathbb{R} \times \mathbb{R} \to \mathbb{R}$ is a locally Lipschitz continuous function, $T$-periodic in the first variable, and, for some $a > 0$, 
$\varphi: (-a,a) \to \mathbb{R}$ is a $C^1$-diffeomorphism satisying $\varphi(0) = 0$ and $\varphi'(s) > 0$ for every $s \in (-a,a)$. As a model example, we can take of course $\varphi(s) = s/\sqrt{1-s^2}$, giving rise to the well known Minkowksi-curvature equation
$$
\left( \frac{u'}{\sqrt{1-(u')^2}}\right)' + f(t,u) = 0,
$$
which has been the object of extensive investigation in the last decades (see \cite{Ma13} for a survey on the topic and \cite{BoFe20} for some more recent references).

Similarly as before, if $\bar{u}(t)$ is a $T$-periodic solution of equation \eqref{eqqua} such that
$\partial_u f(t,\bar{u}(t))$ is well-defined, by $\rho_{\bar{u}}$ we now mean the rotation number of the linear Sturm-Liouville type equation 
$(\varphi'(\bar{u}'(t)) u' )' + \partial_u f(t,\bar{u}(t))u = 0$, meant as the planar Hamiltonian system
\begin{equation}\label{quasrho}
u' = \frac{v}{\varphi'(\bar{u}'(t))}, \qquad v' = - \partial_u f(t,\bar{u}(t))u.
\end{equation}
Again, it is easily seen that $\rho_{\bar{u}} \geq 0$.

As a corollary of Theorem \ref{thmain} and Corollary \ref{corzero}, the following result can be obtained. 

\begin{theorem}\label{thqua}
Assume that there exists a $T$-periodic solution $\bar{u}(t)$ of equation \eqref{eqqua} such that
$$
\rho_{\bar{u}} > 0.
$$
Then, there exists an integer $k^* \geq 1$ such that equation \eqref{eqqua} possesses subharmonic solutions of order $k$, for every $k \geq k^*$. More precisely, for every $k \geq k^*$ and for every integer $j$ coprime with $k$ and satisfying $1 \leq j < k \rho_{\bar u}$, 
equation \eqref{eqqua} possesses at least two non-equivalent subharmonic solutions $u_{k,j}^1(t), u_{k,j}^2(t)$ of order $k$
such that the $u_{k,j}^i(t) - \bar{u}(t)$ has exactly $2j$ zeros on $[0,kT)$ for $i=1,2$.
\end{theorem}

Notice that we do not impose any condition for the function $f$ at infinity; thus Theorem \ref{thqua} can be considered a generalization of \cite[Th. 1]{BoGa13} (in that paper, $\bar{u} \equiv 0$ and an explicit condition is considered instead of $\rho_{\bar{u}} > 0$). 
We also mention that, as proved in \cite[Appendix C]{BoFe20}, the assumption $\rho_{\bar{u}} > 0$ is equivalent to the fact that the
(unique) principal eigenvalue of the $T$-periodic problem for
$$
(\varphi'(\bar{u}'(t)) u' )' + (\lambda + \partial_u f(t,\bar{u}(t))) u = 0
$$
is strictly negative.

\begin{proof}[Proof of Theorem \ref{thqua}]
Let us consider a locally Lipschitz continuous function 
$\tilde f: \mathbb{R} \times \mathbb{R} \to \mathbb{R}$ with the following properties:
\begin{itemize}
\item[(i)] $\tilde f$ is $T$-periodic in the first variable,
\item[(ii)] $\tilde f(t,u) = f(t,u)$ for every $t \in [0,T]$ and $u \in \mathbb{R}$ with $\vert u \vert \leq  L :=\Vert  \bar{u} \Vert_\infty + 2aT$,
\item[(iii)] $\tilde f$ is bounded,
\end{itemize}
(for instance, we can take $\tilde f(t,u) = f(t,\max\{-L,\min\{u,L\}\})$) and, accordingly, we consider the equation 
\begin{equation}\label{eqqua2}
(\varphi(u'))' + \tilde f(t,u) = 0.
\end{equation}
We claim the following:
\begin{itemize}
\item[($\star$)] if $u(t)$ is a $kT$-periodic solution of \eqref{eqqua2} such that $u(t) - \bar{u}(t)$ vanishes somewhere in $[0,kT]$, then
$\vert u(t) \vert < L$ for every $t \in [0,kT]$.
\end{itemize} 
Indeed, the function $x(t) := u(t) - \bar{u}(t)$ is such that $\vert x'(t) \vert < 2a$ for every $t$, and so, since $x(t)$ vanishes at some $t^* \in [0,kT]$, it holds that
$$
\vert u(t) \vert \leq \vert \bar u(t) \vert + \vert x(t) \vert \leq \Vert \bar{u} \Vert_\infty + \left\vert \int_{t^*}^t x'(s) \,ds \right\vert \leq 
\Vert  \bar{u} \Vert_\infty + 2aT = L
$$
for every $t \in [0,kT]$. 
	
As a consequence of $(\star)$, any subharmonic solution $u(t)$ of equation \eqref{eqqua2} with the property that $u(t) - \bar{u}(t)$ vanishes somewhere on $\mathbb{R}$ is also a subharmonic solution of equation \eqref{eqqua}, since $f$ and $\tilde f$ coincide for $\vert u \vert \leq L$. Hence, to prove Theorem \ref{thqua} it is enough to show that the multiplicity result stated therein holds true for the modified equation \eqref{eqqua2}. 
	
Let us set $\bar{v}(t) = \varphi(\bar{u}'(t))$ for $t \in [0,T]$ and define, for $(x,y) \in \mathbb{R}^2$ and $t \in \mathbb{R}$, the Hamiltonian function
$$
H(t,x,y) =\int_0^y \left( \varphi^{-1}(\xi + \bar{v}(t)) -\varphi^{-1}(\bar{v}(t))\right) d\xi+ \int_0^x \left( \tilde f(t,\xi + \bar{u}(t)) - \tilde f(t,\bar{u}(t))\right) d\xi,
$$
giving rise to the planar Hamiltonian system
\begin{equation}\label{hsproof2}
\left\{
\begin{array}{ll}
x' = \varphi^{-1}(y + \bar{v}(t)) - \varphi^{-1}(\bar{v}(t)) \\
y' = -\tilde f(t,x + \bar{u}(t)) + \tilde f(t,\bar{u}(t))\,.
\end{array}
\right.
\end{equation}
It is immediately checked that $H$ satisfies assumption $(H_0)$ and that the linearized system at zero is given by 
$$
\left\{
\begin{array}{ll}
x' = (\varphi^{-1})'(\bar v (t)) y = \frac{y}{\varphi'(\bar u'(t))} \vspace{0.2cm}\\
y' = - \partial_x \tilde f(t,\bar{u}(t))x = - \partial_x  f(t,\bar{u}(t))x.
\end{array}
\right.
$$
Moreover, since $\varphi^{-1}$ and $\tilde f$ are bounded, $H$ satisfies assumption 
$(H_\infty)$ and the linearized system at infinity is given by
$$
\left\{
\begin{array}{ll}
x' = 0 \\
y' = 0.
\end{array}
\right.
$$
Hence, $\rho_\infty = 0$. On the other hand, the linearized system at zero is nothing but system \eqref{quasrho} and so its rotation number is $\rho_{\bar{u}}$. Thus, Corollary \ref{corzero} can be applied, yielding the existence, for $k$ large enough, of $k$-th order subharmonic solutions $(x_{k,j}^i(t),y_{k,j}^i(t))$ (with $i=1,2$) for the Hamiltonian system \eqref{hsproof2}, with winding number on the interval $[0,kT)$ equal to $j \in [1,k \rho_{\bar{u}})$ (here $j$ and $k$ are coprime). Due to $x' = y/\varphi'(\bar{u}'(t))$, the $y$ axis can be crossed only in the clockwise sense, and so the winding number of the path $(x_{k,j}^i(t),y_{k,j}^i(t))$ is half the number of crossing of the $y$-axis, that is, the number of zeros of the function $x_{k,j}^i(t)$ on the interval $[0,kT)$.

Setting $u_{k,j}^i(t) = x_{k,j}^i(t) + \bar{u}(t)$ and $v_{k,j}^i(t) = y_{k,j}^i(t) + \bar{v}(t))$, we thus have that the function $(u_{k,j}^i(t),v_{k,j}^i(t))$ is a subharmonic solution of 
$$
\left\{
\begin{array}{ll}
u' = \varphi^{-1}(v) \\
v' = -\tilde f(t,u),
\end{array}
\right.
$$ 
and so $u_{k,j}^i(t)$ is a subharmonic solution of equation \eqref{eqqua2}. Moreover, by the previous discussion, the number of zeros of $u_{k,j}^i(t) - \bar{u}(t) = x_{k,j}^i(t)$ on the interval $[0,kT)$ is exactly $2j$. As already mentioned, this is enough to conclude that the function $u_{k,j}^i(t)$ is a (subharmonic) solution of equation \eqref{eqqua}, as well.
\end{proof} 

\begin{remark}
\rm
It is worth noticing that, differently from the semilinear case (cf. the proof of Theorem \ref{thsem}), due to the nonlinearity of the differential operator the linear change of variable $x(t) = u(t) - \bar{u}(t)$
does not convert the equation $(\varphi(u'))' + f(t,u) = 0$ into an equation of the same type. However, working at the level of the Hamiltonian system allows us to easily overcome this apparent difficulty.
\end{remark}

\bibliographystyle{plain}
\bibliography{bmref}

\end{document}